\documentclass[12pt]{amsart}
\usepackage{amscd,amsmath,amsthm,amssymb,amsfonts,enumerate,hyperref}

\def\NZQ{\mathbb}

\def\ZZ{{\NZQ Z}}

\def\Bc{{\mathcal B}}

\def\Wc{{\mathcal W}}

\def\opn#1#2{\def#1{\operatorname{#2}}} 
	\opn\chara{char} \opn\length{\ell} \opn\pd{pd} \opn\rk{rk}
	\opn\projdim{proj\,dim} \opn\injdim{inj\,dim} \opn\rank{rank}
	\opn\depth{depth} \opn\grade{grade} \opn\height{height}
	\opn\embdim{emb\,dim} \opn\codim{codim}
	\opn\Cl{Cl}
	
	\opn\Tr{Tr} \opn\bigrank{big\,rank}
	\opn\superheight{superheight}\opn\lcm{lcm}
	\opn\trdeg{tr\,deg}
	\opn\rdeg{rdeg}
	\opn\reg{reg} \opn\lreg{lreg} \opn\ini{in} \opn\lpd{lpd}
	\opn\size{size} \opn\sdepth{sdepth}
	\opn\link{link}\opn\fdepth{fdepth}\opn\lex{lex}
	\opn\tr{tr}
	\opn\type{type}
	\opn\gap{gap}
	\opn\arithdeg{arith-deg}
	\opn\revlex{revlex}
	\opn\div{div} \opn\Div{Div} \opn\cl{cl} \opn\Cl{Cl}
	\opn\Spec{Spec} \opn\Supp{Supp} \opn\supp{supp} \opn\Sing{Sing}
	\opn\Ass{Ass} \opn\Min{Min}\opn\Mon{Mon}
	\opn\Ann{Ann} \opn\Rad{Rad} \opn\Soc{Soc}
	\opn\Im{Im} \opn\Ker{Ker} \opn\Coker{Coker} \opn\Am{Am}
	\opn\Hom{Hom} \opn\Tor{Tor} \opn\Ext{Ext} \opn\End{End}
	\opn\Aut{Aut} \opn\id{id}
	
	\opn\nat{nat}
	\opn\pff{pf}
	\opn\Pf{Pf} \opn\GL{GL} \opn\SL{SL} \opn\mod{mod} \opn\ord{ord}
	\opn\Gin{Gin} \opn\Hilb{Hilb}\opn\sort{sort}
	\opn\PF{PF}\opn\Ap{Ap}
	\opn\mult{mult}
	\opn\bight{bight}
	\opn\div{div}
	\opn\Div{Div}
	\opn\aff{aff}
	\opn\relint{relint} \opn\st{st}
	\opn\lk{lk} \opn\cn{cn} \opn\core{core} \opn\vol{vol}  \opn\inp{inp}
	\opn\nilpot{nilpot}
	\opn\link{link} \opn\star{star}\opn\lex{lex}\opn\set{set}
	\opn\width{wd}
	\opn\Fr{F}
	\opn\QF{QF}
	\opn\G{G}
	\opn\type{type}\opn\res{res}
	\opn\conv{conv}
	\opn\Int{Int}
	\opn\Deg{Deg}
	\opn\Sym{Sym}
	\opn\Con{Con}
	\opn\gr{gr}
	
	\def\pot#1#2{#1[\kern-0.28ex[#2]\kern-0.28ex]}

	\opn\dirlim{\underrightarrow{\lim}}
	\opn\inivlim{\underleftarrow{\lim}}
	\let\to=\rightarrow
	
	\def\Implies{\ifmmode\Longrightarrow \else
		\unskip${}\Longrightarrow{}$\ignorespaces\fi}
	\def\implies{\ifmmode\Rightarrow \else
		\unskip${}\Rightarrow{}$\ignorespaces\fi}
	\def\iff{\ifmmode\Longleftrightarrow \else
		\unskip${}\Longleftrightarrow{}$\ignorespaces\fi}

	\let\:=\colon
	\newtheorem{Theorem}{Theorem}[section]
	\newtheorem{Lemma}[Theorem]{Lemma}
	\newtheorem{Corollary}[Theorem]{Corollary}

	\theoremstyle{definition}
	
	\newtheorem{Definition}[Theorem]{Definition}
	\newtheorem{Example}[Theorem]{Example}
	\newtheorem{Conjecture}[Theorem]{Conjecture}

    \makeatletter
\@namedef{subjclassname@2020}{%
  \textup{2020} Mathematics Subject Classification}
\makeatother

\begin{document}

\title[Bounded powers of edge ideals]{Bounded powers of edge ideals: The strong exchange property}

\author[T.~Hibi]{Takayuki Hibi}
\author[S.~A.~ Seyed Fakhari]{Seyed Amin Seyed Fakhari}

\address{(Takayuki Hibi) Department of Pure and Applied Mathematics, Graduate School of Information Science and Technology, Osaka University, Suita, Osaka 565--0871, Japan}
\email{hibi@math.sci.osaka-u.ac.jp}
\address{(Seyed Amin Seyed Fakhari) Departamento de Matem\'aticas, Universidad de los Andes, Bogot\'a, Colombia}
\email{s.seyedfakhari@uniandes.edu.co}



\begin{abstract}
Let $S=K[x_1, \ldots,x_n]$ denote the polynomial ring in $n$ variables over a field $K$ and $I \subset S$ a monomial ideal.  Given a vector $\mathfrak{c}\in\ZZ_{>0}^n$, the ideal  $I_{\mathfrak{c}}$ is the ideal generated by those monomials belonging to $I$ whose exponent vectors are componentwise bounded above by $\mathfrak{c}$.  Let $\delta_{\mathfrak{c}}(I)$ be the largest integer $q$ for which $(I^q)_{\mathfrak{c}}\neq 0$.  Let $I(G) \subset S$ denote the edge ideal of a finite graph $G$ on the vertex set $V(G) = \{x_1, \ldots, x_s\}$.  In our previous work, it is shown that $(I(G)^{\delta_{\mathfrak{c}}(I)})_{\mathfrak{c}}$ is a polymatroidal ideal.  Let $\mathcal{W}(\mathfrak{c},G)$ denote the minimal system of monomial generators of $(I(G)^{\delta_{\mathfrak{c}}(I)})_{\mathfrak{c}}$.  It follows that $\mathcal{W}(\mathfrak{c},G)$ satisfies the symmetric exchange property.  In the present paper, the question when $\mathcal{W}(\mathfrak{c},G)$ enjoys the strong exchange property, or equivalently, when $\mathcal{W}(\mathfrak{c},G)$ is of Veronese type is studied.
\end{abstract}

\maketitle

\section*{Introduction}
Let $S=K[x_1, \ldots,x_n]$ denote the polynomial ring in $n$ variables over a field $K$ and $I \subset S$ a monomial ideal.  Let $\ZZ_{>0}$ denote the set of positive integers.  Given $\mathfrak{c}=(c_1, \ldots, c_n) \in\ZZ_{>0}^n$, the ideal $I_{\mathfrak{c}} \subset S$ is the ideal generated by those monomials $x_1^{a_1} \cdots x_n^{a_n}$ belonging to $I$ with $a_i \leq c_i$, foe each $i=1, \ldots, n$.  Let $\delta_{\mathfrak{c}}(I)$ be the largest integer $q$ for which $(I^q)_{\mathfrak{c}}\neq 0$.  

Let $G$ be a finite graph with no loop, no multiple edge and no isolated vertex on the vertex set $V(G)=\{x_1, \ldots, x_n\}$ and $E(G)$ the set of edges of $G$.  The {\em edge ideal} of $G$ is the ideal $I(G) \subset S$ generated by those $x_ix_j$ with $\{x_i, x_j\} \in E(G)$.  Let $\Wc(\mathfrak{c},G) = \{w_1, \ldots, w_s\}$ denote the minimal set of monomial generators of $(I(G)^{\delta_{\mathfrak{c}}(I(G))})_{\mathfrak{c}}$ and $\mathcal{B}(\mathfrak{c},G)$ the toric ring $K[w_1, \ldots, w_s]\subset S$.  In \cite{HSF1}, it is proved that $(I(G)^{\delta_{\mathfrak{c}}(I)})_{\mathfrak{c}}$ is a polymatroidal ideal. It then follows from \cite[Corollary 6.2]{HH_discrete} that $\mathcal{B}(\mathfrak{c},G)$ is normal and Cohen--Macaulay.  In \cite{HSF2} the question when $\mathcal{B}(\mathfrak{c},G)$ is Gorenstein is studied and especially it is shown that $\mathcal{B}(\mathfrak{c},G)$ is Gorenstein for all $\mathfrak{c} \in\ZZ_{>0}^n$ if and only if there is an integer $t >2$ for which every connected component of $G$ is either $K_2$ or $K_t$, where $K_t$ is the complete graph on $t$ vertices.  

Let $T = K[z_1, \ldots, z_s]$ denote the polynomial ring in $s$ variables over a field $K$ and define the surjective ring homomorphism $\pi^\mathfrak{c}_G : T \to \mathcal{B}(\mathfrak{c},G)$ by setting $\pi_G^\mathfrak{c}(z_i) = w_i$ for $1 \leq i \leq s$.  The {\em toric ideal} of $\mathcal{B}(\mathfrak{c},G)$ is the kernel $\Ker(\pi_G^\mathfrak{c})$ of $\pi_G^\mathfrak{c}$.  Since $(I(G)^{\delta_{\mathfrak{c}}(I)})_{\mathfrak{c}} = (w_1, \ldots, w_s)$ is polymatroidal, it follows from \cite[Theorem 4.1]{HH_discrete} that  $\Wc(\mathfrak{c},G) = \{w_1, \ldots, w_s\}$ satisfies the symmetric exchange property.  In other words, if $w_i = x_1^{a_1} \cdots x_n^{a_n}$ and $w_j = x_1^{b_1} \cdots x_n^{b_n}$ belong to $\Wc(\mathfrak{c},G)$ with $a_\xi > b_\xi$, then there is $\rho$ with $a_{\rho} < b_{\rho}$ for which both $x_\rho(w_i/x_\xi)$ and $x_\xi(w_j/x_\rho)$ belong to $\Wc(\mathfrak{c},G)$.  Let $w_{i_0}=x_\rho(w_i/x_\xi)$ and $w_{j_0}=x_\xi(w_j/x_\rho)$.  Then $z_iz_j - z_{i_0}z_{j_0}$ belongs to $\Ker(\pi_G^\mathfrak{c})$.  One calls $z_iz_j - z_{i_0}z_{j_0}$ a {\em symmetric exchange binomial} of $\Ker(\pi_G^\mathfrak{c})$. 

On the other hand, we say that $\Wc(\mathfrak{c},G) = \{w_1, \ldots, w_s\}$ enjoys the {\em strong exchange property} if, for all $w_i = x_1^{a_1} \cdots x_n^{a_n}$ and $w_j = x_1^{b_1} \cdots x_n^{b_n}$ belonging to $\Wc(\mathfrak{c},G)$ and for all $\xi$ and $\rho$ with $a_\xi > b_\xi$ and $a_{\rho} < b_{\rho}$, one has $x_\rho(w_i/x_\xi) \in \Wc(\mathfrak{c},G)$.  It follows from \cite[Theorem 5.3 (b)]{HH_discrete} that 

\begin{Theorem}[\cite{HH_discrete}]
If $\Wc(\mathfrak{c},G) = \{w_1, \ldots, w_s\}$ enjoys the strong exchange property, then $\Ker(\pi_G^\mathfrak{c})$ possesses a quadratic Gr\"obner basis and is generated by all symmetric exchange binomials of $\Ker(\pi_G^\mathfrak{c})$.
\end{Theorem}

Let $\mathfrak{a}=(a_1, \ldots, a_n) \in\ZZ_{>0}^n$ and $d \in \ZZ_{>0}$ with $d\leq a_1+\cdots +a_n$.  Also, let ${\bf V}^{(d)}_n(\mathfrak{a})$ denote the minimal set of monmomial generators of $((x_1, \ldots, x_n)^d)_\mathfrak{a}$.  Recall from \cite{DH} that the {\em algebra of Veronese type} $A(d;\mathfrak{a})$ is the toric ring generated by all monomials belonging to ${\bf V}^{(d)}_n(\mathfrak{a})$. A {\em star graph} on $n+1$ vertices is the finite graph $Q_n$ on $V(Q_n)=\{x_1, \ldots, x_n, x_{n+1}\}$ with $E(Q_n) = \{\{x_i, x_{n+1}\}: 1 \leq i \leq n\}$. Set $(\mathfrak{a},d):=(a_1, \ldots, a_n, d) \in\ZZ_{>0}^{n+1}$. One has $\delta_{(\mathfrak{a},d)}(I(Q_n)) = d$. Clearly, $\Wc((\mathfrak{a},d),Q_n)=x_{n+1}^d \cdot {\bf V}^{(d)}_{n}(\mathfrak{a})$ and $\Bc((\mathfrak{a},d),Q_n) \cong A(d;\mathfrak{a})$.  We say that $\Wc(\mathfrak{c},G)$ is {\em of Veronese type} if $\Wc(\mathfrak{c},G)$ is of the form $w \cdot {\bf V}^{(d_0)}_{n_0}(\mathfrak{a'})$, where $n_0, d_0 \in \ZZ_{>0}$, $\mathfrak{a}' =(a_1', \ldots, a_{n_0}')\in \ZZ_{>0}^{n_0}$ with $d\leq a_1'+\cdots +a_{n_0}'$ and $w$ is a monomial. Now, it follows from \cite[Theorem 1.1]{HHV} that 

\begin{Theorem}[\cite{HHV}]
The minimal set $\Wc(\mathfrak{c},G)$ of monomial generators of $\mathcal{B}(\mathfrak{c},G)$ enjoys the strong exchange property if and only if $\Wc(\mathfrak{c},G)$ is of Veronese type. 
\end{Theorem}

Given a finite graph $G$ on the vertex set $V(G) = \{x_1, \ldots, x_n\}$, one can find $\mathfrak{c}=(c_1, \ldots, c_n) \in \ZZ_{>0}^n$ for which $\Wc(\mathfrak{c},G)$ is of Veronese type.  In fact, if $\mathfrak{c}=(c_1, \ldots, c_n) \in\ZZ_{>0}^n$ is the exponet vector of the monomial $u = \prod_{\{x_i,x_j\}\in E(G)} x_ix_j \in S$.  Then $\Wc(\mathfrak{c},G) = \{u\}$ and $\Wc(\mathfrak{c},G)$ is of Veronese type.  On the other hand, in proof of \cite[Theorems 4.5]{HSF2}, it is remarked that if either $2\delta_{\mathfrak{c}}(I(G)) = c_1 + \cdots + c_n$ or $2\delta_{\mathfrak{c}}(I(G)) = c_1 + \cdots + c_n - 1$, then $\mathcal{B}(\mathfrak{c},G)$ is the polynomial ring and $\Wc(\mathfrak{c},G)$ is of Veronese type.

\begin{Definition}
  We say that a finite graph $G$ on $n$ vertices is {\em of Veronese type} if $\mathcal{W}(\mathfrak{c},G)$ is of Veronese type for all $\mathfrak{c} \in \ZZ_{>0}^n$.  
  
  Equivalently, we say that a finite graph $G$ on $n$ vertices {\em enjoys the strong exchange property} if $\mathcal{W}(\mathfrak{c},G)$ enjoys the strong exchange property for all $\mathfrak{c} \in \ZZ_{>0}^n$.   
\end{Definition}

With taking into account of the most attractive research problems \cite[p.~241]{HH_discrete}, one can naturally ask if, for all finite graphs $G$ on $n$ vertices and for all $\mathfrak{c} \in \ZZ_{>0}^n$, the toric ideal $\Ker(\pi_G^\mathfrak{c})$ possesses a quadratic Gr\"obner basis and is generated by all symmetric exchange binomials of $\Ker(\pi_G^\mathfrak{c})$. As one of the most fundamental steps for this question, in the present paper, we mainly classify cycles, trees and unicyclic graphs which enjoy the strong exchange property.        

After summarizing fundamental notion and terminologies in Section $1$, in Section $2$, we recall \cite[Theorem 4.1]{HSF2}, which guarantees that, for the complete multipartite graph $K_{n_1, \ldots, n_m}$ of type $(n_1, \ldots, n_m)$ and a matching $M$ of $K_{n_1, \ldots, n_m}$, the finite graph $K_{n_1, \ldots, n_m} - M$ is of Veronese type.  Thus, in particular, $K_{n_1, \ldots, n_m} - M$ enjoy the strong exchange property.  In Sections $3, 4$ and $5$,  we classify cycles, trees and unicyclic graphs enjoying the strong exchange property.  The cycle $C_n$ of length $n \geq 3$ enjoys the strong exchange property if and only if $3 \leq n \leq 7$ (Theorem \ref{maincyc}).  Our classification of trees and unicyclic graphs 
are summarized in Theorems \ref{treeclass} and \ref{classification_unicyclic}.  In our classification, Lemma \ref{trianind} saying that every triangle-free graph with independence number at most $3$ enjoys the strong exchange property and  Lemma \ref{treedel} showing that if $G$ enjoys the strong exchange property and if $x$ is a leaf of $G$, then $G-x$ enjoys the strong exchange property are indispensable.

\section{Preliminaries}
We summarize notations and terminologies on finite graphs. Let $G$ be a finite graph with no loop, no multiple edge and no isolated vertex on the vertex set $V(G)=\{x_1, \ldots, x_n\}$ and $E(G)$ the set of edges of $G$.   

\begin{itemize}
 \item 
We say that $x_i \in V(G)$ is {\em adjacent} to $x_j \in V(G)$ in $G$ if $\{x_i,x_j\} \in E(G)$.  In addition, $x_j$ is called a {\em neighbor} of $x_i$. Let $N_G(x_i)$ denote the set of vertices of $G$ to which $x_i$ is adjacent. The cardianlity of $N_G(x_i)$ is the {\em degree} of $x_i$, denoted by ${\rm deg}_G(x_i)$. A {\em leaf} of $G$ is a vertex of degree one. Furthermore, if $A \subset V(G)$, then we set $N_G(A) := \cup_{x_i \in A} N_G(x_i)$.
\item 
We say that $e \in E(G)$ is {\em incident} to $x \in V(G)$ if $x \in e$.
 
 \item   
A {\em tree} is a finite connected graph with no cycle.
    \item 
A {\em triangle} is the cycle of length $3$.  A {\em triangle-free graph} is a finite graph with no triangle.
    \item 
A {\em unicyclic} graph is a finite connected graph having a unique cycle.
 \item 
 A subset $C \subset V(G)$ is called {\em independent} if $\{x_i, x_j\} \not\in E(G)$ for all $x_i, x_j \in C$ with $x_i \neq x_j$.  The {\em independence number} of $G$ is the biggest cardinality of independent sets of $G$.  
    \item 
A {\em matching} of $G$ is a subset $M \subset E(G)$ for which $e \cap e' = \emptyset$ for $e, e' \in M$ with $e \neq e'$.      
    \item 
If $M$ is a matching of $G$, then we define $G-M$ to be the finite graph obtained from $G$ by removing all edges belonging to $M$.  
\item 
If $U \subset V(G)$, then $G -  U$ is the finite graph on $V(G)\setminus U$ with $E(G - U) = \{e \in E(G) : e \cap U = \emptyset\}$.  In other words, $G - U$ is the {\em induced subgraph} $G_{V(G)\setminus U}$ of $G$ on $V(G)\setminus U$.     
    \item 
In the polynomial ring $S = K[x_1, \ldots, x_n]$, unless there is a misunderstanding, for an edge $e = \{x_i, x_j\}$, we employ the notation $e$ instead of the monomial $x_ix_j\in S$.  For example, if $e_1 = \{x_1, x_2\}$ and $e_2 = \{x_2, x_5\}$, then $e_1^2e_2 = x_1^2x_2^3x_5$. 
\end{itemize}

\section{Complete multipartite graphs}
Let $m \geq 2, n_1 \geq 1, \ldots, n_m \geq 1$ be integers and  
\[
V_i = \{x_{\sum_{j=1}^{i-1} n_j+1}, \ldots, x_{\sum_{j=1}^{i} n_j}\}, \, \, \, \, \, \, \, \, \, \, 1 \leq i \leq m.
\]
The finite graph $K_{n_1, \ldots, n_m}$ on $V(K_{n_1, \ldots, n_m}) = V_1 \sqcup \cdots \sqcup V_m$ with 
\[
E(K_{n_1, \ldots, n_m}) = \{ \{ x_k, x_\ell \} : x_k \in V_i, \, x_\ell \in V_{j}, \, 1 \leq i < j \leq m\}.
\]
is called the {\em complete multipartite graph} \cite[p.~394]{compressed} of type $(n_1, \ldots, n_m)$. 

\begin{Theorem}
\label{completeSEP}
Let $K_{n_1, \ldots, n_m}$ be the complete multipartite graph 
and $M$ a matching of $K_{n_1, \ldots, n_m}$ such that the graph $G:=K_{n_1, \ldots, n_m} - M$ has no isolated vertex. Then $G$ enjoys the strong exchange property.  
\end{Theorem}

\begin{proof}
The desired result follows immediately from \cite[Theorem 4.1]{HSF2} and its proof which guarantees that $K_{n_1, \ldots, n_m} - M$ is of Veronese type. 
\end{proof}

\section{Cycles}
Let $C_n$ denote the cycle of length $n$ on $V(C_n) = \{x_1, \ldots, x_n\}$ with $E(C_n) = \{\{x_1,x_2\}, \{x_2,x_3\}, \ldots, \{x_{n-1},x_n\},\{x_n, x_1\}\}$.  We classify the cycles enjoying the strong exchange property.  

\begin{Lemma} \label{cycnot}
The cycle $C_n$ with $n \geq 8$ does not enjoy the strong exchange property.
\end{Lemma}

\begin{proof}
Consider the vector $\mathfrak{c}$ defined as follows. 
$$c_i=
\left\{
	\begin{array}{ll}
		2  &  i=1,3,7,\\
		1 & {\rm otherwise}
	\end{array}
\right.$$
Set $\delta:=\delta_{\mathfrak{c}}(I(C_n))$.

\medskip

{\bf Claim.} $\delta=\lceil\frac{n}{2}\rceil$.

\medskip

\noindent
{\it Proof of Claim.} First, assume that $n$ is even. Then $C_n$ is a bipartite graph on the vertex set $X_1\sqcup X_2$ with $X_1=\{x_1,x_3, \ldots, x_{n-1}\}$ and $X_2=\{x_2,x_4, \ldots, x_n\}$. Since$$\sum_{x_i\in X_2}c_i=\frac{n}{2},$$it follows that $\delta\leq \frac{n}{2}$. On the other hand,$$(x_1x_2)(x_3x_4)\cdots (x_{n-1}x_n)$$is a $\mathfrak{c}$-bounded monomial of degree $n$. Thus, $\delta=\frac{n}{2}$.

Next, assume that $n$ is odd. The monomial$$(x_1x_2)(x_3x_4)(x_5x_6)\cdots (x_{n-2}x_{n-1})(x_1x_n)$$shows that $\delta\geq \frac{n+1}{2}$. To prove the reverse inequality, let $u$ be a monomial in the minimal set of monomial generators of $(I(C_n)^{\delta})_{\mathfrak{c}}$. Note that $C_n-x_n$ is a bipartite graph on the vertex set $X_1\sqcup X_2$ with $X_1=\{x_1,x_3, \ldots, x_{n-2}\}$ and $X_2=\{x_2,x_4, \ldots, x_{n-1}\}$. Moreover,$$\sum_{x_i\in X_2}c_i=\frac{n-1}{2}.$$Thus, if $u$ is not divisibe by $x_n$, then it follows from the above equality that ${\rm deg}(u)\leq n-1$. If $u$ is divisible by $x_n$, then there is $\ell\in\{1,n-1\}$ such that $u=(x_{\ell}x_n)v$, for some monomial $v\in I(G)^{\delta-1}$ which is not divisible by $x_n$. Again, using the above equality, we conclude that ${\rm deg}(v)\leq n-1$. Thus, ${\rm deg}(u)\leq n+1$. This completes the proof of Claim.

\medskip

Assume that $n$ is even. Consider two monomials$$w_1=(x_1x_2)(x_3x_4)(x_5x_6)\cdots (x_{n-3}x_{n-2})(x_1x_n)$$and$$w_2=(x_2x_3)(x_3x_4)(x_6x_7)(x_7x_8)(x_9x_{10})\cdots (x_{n-1}x_n)$$in $\Wc(\mathfrak{c},C_n)$. Obviously, ${\rm deg}_{x_3}(w_2)> {\rm deg}_{x_3}(w_1)$ and ${\rm deg}_{x_5}(w_1)>{\rm deg}_{x_5}(w_2)$. If $\Wc(\mathfrak{c},C_n)$ enjoys the strong exchange property, then the monomial $x_3w_1/x_5$ must belong to $\Wc(\mathfrak{c},C_n)$ which is impossible, as this monomial is divisible by $x_1^2x_2x_3^2$.

Finally, assume that $n$ is odd. Then consider two monomials$$w_1=(x_1x_2)(x_3x_4)(x_5x_6)\cdots (x_{n-2}x_{n-1})(x_1x_n)$$and$$w_2=(x_2x_3)(x_3x_4)(x_6x_7)(x_7x_8)(x_9x_{10})\cdots (x_{n-2}x_{n-1})(x_1x_n)$$in $\Wc(\mathfrak{c},C_n)$. Obviously, ${\rm deg}_{x_3}(w_2)> {\rm deg}_{x_3}(w_1)$ and ${\rm deg}_{x_5}(w_1)>{\rm deg}_{x_5}(w_2)$. If $\Wc(\mathfrak{c},C_n)$ enjoys the strong exchange property, then the monomial $x_3w_1/x_5$ must belong to $\Wc(\mathfrak{c},C_n)$ which is impossible, as this monomial is divisible by $x_1^2x_2x_3^2$.
\end{proof}

Now, in order to show that $C_4, C_5, C_6, C_7$ enjoy the strong exchange property, we prove a much stronger result (Lemma \ref{trianind}).   

\begin{Lemma} \label{polmat}
Let $K[x,y,z]$ be the polynomial ring in three variables  and let $I\subset K[x,y,z]$ be a polymatroidal ideal. Then the minimal set of monomial generators of $I$ enjoys the strong exchange property.
\end{Lemma}

\begin{proof}
Let $G(I)$ denote the minimal set of monomial generators of $I$. Consider two monomials $w_1=x^ay^bz^c$ and $w_2=x^{a'}y^{b'}z^{c'}$ belonging to $G(I)$.  Assume that $a>a'$ and $b<b'$. We must show that $x^{a-1}y^{b+1}z^c\in I$.  On the contrary,  assume that $x^{a-1}y^{b+1}z^c\notin I$. If $c\geq c'$, then since $I$ is a polymatroidal ideal, one has $x^{a-1}y^{b+1}z^c\in I$, a contradiction.  Thus, $c<c'$. Hence, $x^{a'+1}y^{b'}z^{c'-1}\in I$.  Since  $a\geq a'+1, c\leq c'-1$ and $b<b'$, one has $a>a'+1$. It follows from $x^{a-1}y^{b+1}z^c\notin I$ that $c<c'-1$ and $x^{a'+2}y^{b'}z^{c'-2}\in I$.  Since  $a\geq a'+2, c\leq c'-2$ and $b<b'$, one has $a > a'+2$.  It follows from $x^{a-1}y^{b+1}z^c\notin I$ that $c<c'-2$ and $x^{a'+3}y^{b'}z^{c'-3}\in I$.  Continuing these processes yields a contradiction and $x^{a-1}y^{b+1}z^c\in I$, as desired.
\end{proof}

\begin{Lemma} \label{trianind}
Every triangle-free graph $G$ with independence number at most 3 enjoys the strong exchange property.    
\end{Lemma}

\begin{proof}
Let $V(G)=\{x_1, \ldots, x_n\}$ and $\mathfrak{c} = (c_1, \ldots, c_n)\in \ZZ_{>0}^n$. We show that $\Wc(\mathfrak{c},G)$ enjoys the strong exchange property. If for some vertex $x_j\in V(G)$, we have $c_j> \sum_{x_t\in N_G(x_j)}c_t$, then $\Wc(\mathfrak{c},G)=\Wc(\mathfrak{c'},G)$, where $\mathfrak{c'}\in \ZZ_{>0}^n$ is the vector obtained from $\mathfrak{c}$ by replacing $c_j$ with $\sum_{x_t\in N_G(x_j)}c_t$. So, from the beginning, we assume that $c_j\leq \sum_{x_t\in N_G(x_j)}c_t$, for each $j=1, \ldots, n$. 

Set $\delta:=\delta_{\mathfrak{c}}(I(G))$. If $2\delta\geq (c_1+\cdots +c_n)-1$, then $\Wc(\mathfrak{c},G)$ enjoys the strong exchange property.  Suppose that $2\delta\leq (c_1+\cdots +c_n)-2$.  Let $v=x_1^{a_1}\cdots x_n^{a_n}=e_1\cdots e_{\delta}\in \Wc(\mathfrak{c},G)$, where $e_1, \ldots, e_{\delta}$ are edges of $G$. If there is an edge $\{x_i,x_j\}$ of $G$ with $a_i\leq c_i-1$ and $a_j\leq c_j-1$, then $(x_ix_j)v$ is a $\mathfrak{c}$-bounded monomial in $I(G)^{\delta+1}$ which is a contradiction.  Thus, the set $$A_v=\{x_i \in V(G): a_i\leq c_i-1\}$$is an independent set of $G$.  Note that $a_t=c_t$ for each $x_t\not \in  A_v$.  It follows from our assumption that $1\leq |A_v|\leq 3$.  In what follows each of the cases $|A_v|=3, |A_v|=2$ and $|A_v|=1$ is discussed separately.  

\medskip

{\bf (Case 1)} Let $|A_v|=3$. 

\noindent
Suppose that in the representation of $v=e_1\cdots e_{\delta}$, there is an edge, say, $e_1=\{x_p,x_q\}$ with $e \cap A_v = \emptyset$.  Since $A_v$ is a maximal independent set of $G$, there are vertices $x_i,x_j\in A_v$ with $\{x_i,x_p\}, \{x_j,x_q\}\in E(G)$. Since $G$ is a triangle-free graph, one has $x_i\neq x_j$. Thus, $$(x_ix_j)v=(x_ix_p)(x_jx_q)e_2\cdots e_{\delta}\in I(G)^{\delta+1},$$contradicting the definition of $\delta$. This contradiction implies that each of the edges $e_1, \ldots, e_{\delta}$ is incident to exactly one vertex of $A_v$. Therefore,
\[
\delta=\sum_{x_i\in A_v}a_i=\sum_{x_i\not\in A_v}a_i=\sum_{x_i\notin A_v}c_i.
\]
Let $u=f_1 \cdots f_{\delta} \in \Wc(\mathfrak{c},G)$ with each $f_j \in E(G)$.  Since $A_v$ is independent, each $f_i$ is incident to at most one vertex in $A_v$.  In other words, each $f_i$ is incident to at least on vertex in $V(G) \setminus A_v$. Since the number of edges appearing in the representation of $u$ is $\delta=\sum_{x_i\notin A_v}c_i$, it follows that $f_i \cap A_v \neq \emptyset$ for each $1 \leq i \leq \delta$ and that
    for each $x_i \not\in A_v$ the number of edges appearing in the representation of $u=f_1 \cdots f_{\delta}$ which are incident to $x_i$ is $c_i$.
Consequently, every monomial $u\in \Wc(\mathfrak{c},G)$ is of the form$$u=u'\prod_{x_i\notin A_v}x_i^{c_i},$$where $u'$ is a monomial on the variables belonging to $A_v$. Thus, $$(I(G)^{\delta})_{\mathfrak{c}}=J\prod_{x_i\notin A_v}x_i^{c_i},$$ where $J$ is a polymatroidal ideal in three variables. Now, Lemma \ref{polmat} guarantees that  $\Wc(\mathfrak{c},G)$ enjoys the strong exchange property.

\medskip

{\bf (Case 2)} Let $|A_v|=2$.

\smallskip

\noindent
{\bf (Subcase 2.1)} Suppose that 
for each $x_k\notin A_v$, the set $A_v\cup\{x_k\}$ is not an independent set of $G$. In other words, $A_v$ is a maximal independent set of $G$. Assume that in the representation of $v=e_1\cdots e_{\delta}$, there is an edge, say, $e_1=\{x_p,x_q\}$ which is incident to no vertex of $A_v$. Then by the same argument as in the proof of Case 1, we derive a contradiction.  Hence, each of the edges $e_1, \ldots, e_{\delta}$ is incident to exactly one vertex of $A_v$. Thus,$$\delta=\sum_{x_i\in A_v}a_i=\sum_{x_i\not\in A_v}a_i=\sum_{x_i\notin A_v}c_i.$$Consequently, the similar discussion as in (Case 1) implies the assertion.  

\smallskip

\noindent
{\bf Subcase 2.2.} Suppose that there is 
$x_k\notin A_v$ for which $A_v\cup\{x_k\}$ is an independent set of $G$. Let $A_v=\{x_i,x_j\}$ and set  $A_v':=A_v\cup \{x_k\}$. It follows from the assumption that $A_v'$ is a maximal independent set of $G$. 

\medskip

{\bf Claim 1.} In the representation of  $v=e_1\ldots e_{\delta}$, for any pair of edges $e',e''$ which are incident to none of $x_i,x_j, x_k$, one has $e'\cap e''\neq\emptyset$. 

\medskip

\noindent
{\it Proof of Claim 1.} Suppose that $e'=\{x_{p'},x_{q'}\}$ and $e''=\{x_{p''},x_{q''}\}$. On the contrary, assume that $e'\cap e''=\emptyset$. If $\{x_{p'}, x_i\}, \{x_{q'},x_j\}\in E(G)$, then$$(x_ix_j)v=(x_ix_{p'})(x_jx_{q'})v/e'\in (I(G)^{\delta+1})_{\mathfrak{c}},$$which is a contradiction. Similarly, if $\{x_{p'},x_j\}, \{x_{q'},x_i\}\in E(G)$, we derive a contradiction. Therefore, as $G$ is triangle-free, at least one of the vertices $x_{p'}, x_{q'}$ is  adjacent to neither $x_i$ nor $x_j$.  Let $\{x_{p'},x_i\}, \{x_{p'},x_j\}\notin E(G)$. Since $A_v'$ is a maximal independent set of $G$ with $x_{p'}\notin A_v'$, we deduce that $\{x_{p'},x_k\}\in E(G)$. Similarly, we assume that $\{x_{p''},x_i\}, \{x_{p''},x_j\}\notin E(G)$ and $\{x_{p''},x_k\}\in E(G)$. As $G$ is a triangle-free graph, we deduce that $\{x_{p'},x_{p''}\}\notin E(G)$. Therefore, $\{x_{p'}, x_{p''}, x_i,x_j\}$ is an independent set of $G$ of size four, which is a contradiction. This proves Claim 1.

\medskip

Since $G$ is a triangle-free graph, it follows from Claim 1 that there is a vertex $x_{\ell}$ for which in the representation of  $v=e_1\ldots e_{\delta}$, each edge $e_s$ which is incident to none of $x_i,x_j, x_k$, is incident to $x_{\ell}$. In other words, each of $e_1, \ldots, e_{\delta}$ is incident to at least one of $x_i,x_j, x_k, x_{\ell}$. 

\medskip

{\bf Claim 2.} We may choose $x_{\ell}$ satisfying $\{x_i,x_{\ell}\}, \{x_j,x_{\ell}\}\notin E(G)$.

\medskip

\noindent
{\it Proof of Claim 2.} Assume that $\{x_{\ell},x_i\}\in E(G)$ (the case $\{x_{\ell},x_j\}\in E(G)$ can be handled similarly). In the representation of $v$, suppose that the edges $e_1, \ldots, e_h$ are incident to none of the vertices $x_i,x_j,x_k$. In particular, they are incident to $x_{\ell}$. Let $e_s=\{x_{\ell},x_{\ell_s}\}$ for $s=1, \ldots,h$.  Since $G$ is triangle-free graph, it follows that $\{x_i,x_{\ell_s}\}\notin E(G)$. If $\{x_j,x_{\ell_s}\}\in E(G)$, then$$x_ix_jv=(x_ix_{\ell})(x_jx_{\ell_s})v/e_s\in (I(G)^{\delta+1})_{\mathfrak{c}},$$a contradiction. Thus, $\{x_j,x_{\ell_s}\}\notin E(G)$. Consequently, the set$$\{x_i,x_j, x_{\ell_1}, \ldots, x_{\ell_h}\}$$ is an independent set of $G$. Since the independent number of $G$ is at most $3$, we deduce that $x_{\ell_1}=\ldots =x_{\ell_h}$. In other words, all edges $e_1, \ldots, e_h$ are the same, and all are incident to $x_{\ell_1}$. Replacing $x_{\ell}$ with $x_{\ell_1}$ proves Claim 2.

\medskip

Since $A_v'$ is a maximal independent set of $G$, Claim 2 says that $\{x_k, x_{\ell}\}\in E(G)$. 

\medskip

{\bf Claim 3.} Assume that in the representation of $v=e_1\cdots e_{\delta}$, there are two edges $e_r,e_{r'}$ which are  incident to none of $x_i,x_j$. Then either both $e_r,e_r'$ are incident to $x_k$ or both $e_r,e_r'$ are incident to $x_{\ell}$. 

\medskip

\noindent
{\it Proof of Claim 3.} On the contrary, suppose that $x_k \not\in e_{r}$ and $x_\ell \not\in e_{r'}$.  Thus, $x_\ell \in e_{r}$ and $x_k \in e_{r'}$.  Let $e_r=\{x_{\ell},x_{\ell'}\}$ and $e_{r'}=\{x_k,x_{k'}\}$.  Since $G$ is  triangle-free, $\{x_k,x_{\ell'}\}, \{x_{k'},x_{\ell}\}\notin E(G)$.  In particular, $x_{k'}\neq x_{\ell'}$. Consider the set $\{x_i,x_j, x_k, x_{\ell'}\}$. Since the independent number of $G$ is at most three, either $\{x_{\ell'},x_i\}\in E(G)$ or $\{x_{\ell'},x_j\}\in E(G)$.  Without loss of generality, we may assume that  $\{x_{\ell'},x_i\}\in E(G)$.  Similarly, by considering the set $\{x_i,x_j, x_{k'}, x_{\ell}\}$ and using Claim 2, we deduce that either $\{x_{k'},x_i\}\in E(G)$ or $\{x_{k'},x_j\}\in E(G)$. Assume that $\{x_{k'},x_j\}\in E(G)$. This implies that$$(x_ix_j)v=(x_ix_{\ell'})(x_jx_{k'})(x_kx_{\ell})v/(e_re_{r'})\in (I(G)^{\delta+1})_{\mathfrak{c}},$$a contradiction. This contradiction shows that $\{x_{k'},x_j\}\notin E(G)$. Consequently, $\{x_{k'},x_i\}\in E(G)$. Recall from the first paragraph of the proof that$$c_j\leq  \sum_{x_t\in N_G(x_j)}c_t.$$Thus,$$a_j< c_j\leq \sum_{x_t\in N_G(x_j)}c_t=\sum_{x_t\in N_G(x_j)}a_t.$$Therefore, in the representation of $v=e_1\ldots, e_{\delta}$, there is an edge $e_{r''}$ which is incident to a vertex $x_t\in N_G(x_j)$ but not to $x_j$. Assume that $e_{r''}=\{x_t, x_{t'}\}$. It follows from Claim 2 and the fact that $\{x_i,x_j, x_k\}$ is independent that $x_{t'}\in\{x_i, x_j, x_k, x_{\ell}\}$. However, $x_{t'}\neq x_j$, as  $e_{r''}$ is not incident to $x_j$. If $x_{t'}=x_i$, then$$(x_ix_j)v=(x_jx_t)(x_ix_{k'})(x_ix_{\ell'})(x_kx_{\ell})v/(e_re_{r'}e_{r''})\in (I(G)^{\delta+1})_{\mathfrak{c}},$$a contradiction. If $x_{t'}=x_k$, then$$(x_ix_j)v=(x_jx_t)(x_ix_{\ell'})(x_kx_{\ell})v/(e_re_{r''})\in (I(G)^{\delta+1})_{\mathfrak{c}},$$a contradiction. Similarly, if $x_{t'}=x_{\ell}$, one derives a contradiction. This proves our Claim 3.

\medskip

It follows from Claim 3 that either each of the edges $e_1, \ldots, e_{\delta}$ is incident to  one of the vertices $x_i,x_j, x_k$ or each of them are incident to one of $x_i,x_j,x_{\ell}$. Assume the first case happens (the second case can be handled similarly). Thus,$$\delta=a_i+a_j+a_k=\sum_{x_t\notin \{x_i,x_j,x_k\}}a_t=\sum_{x_t\notin \{x_i,x_j,x_k\}}c_t.$$Consequently, as discussed in Case 1, every monomial $u\in \Wc(\mathfrak{c},G)$, has the form$$u=u'\prod_{x_t\notin \{x_i,x_j,x_k\}}x_t^{c_t},$$where each $u'$ is a monomial on the $x_i,x_j, x_k$. Thus,$$(I(G)^{\delta})_{\mathfrak{c}}=J\prod_{x_t\notin \{x_i,x_j,x_k\}}x_t^{c_t},$$
where $J$ is a polymatroidal ideal in three variables. Now, Lemma \ref{polmat} guarantees that  $\Wc(\mathfrak{c},G)$ enjoys the strong exchange property.

\medskip

{\bf (Case 3)} Let $|A_v|=1$. 

\noindent
Let $A_v=\{x_k\}$.  Recall that, by the definition of $A_v$, one has $a_t=c_t$ for each $x_t \not= x_k$.  Since $2\delta\leq (c_1+\cdots +c_n)-2$, it follows that $a_k\leq c_k-2$. Let $x_{k'}$ be a neighbor of $x_k$ in $G$. One has $a_{k'}=c_{k'}\geq 1$. Thus, in the representation of $v=e_1\cdots e_{\delta}$, there is an edge, say, $e_1$ with $x_{k'} \in e_1$. If $x_k \not\in e_1$, then $e_1=\{x_{k'},x_{k''}\}$ with $x_{k''}\neq x_k$.  Replacing $v$ by $v':=x_kv/x_{k''}=(x_kx_{k'})e_2\cdots e_{\delta}$ and noting that $A_{v'}=\{x_k, x_{k''}\}$, we are reduced to Case 2. Thus, we may assume that, in the representation of $v$, if an edge $e_i$ is incident to a neighbor of $x_k$, it is incident to $x_k$ too. In particular,$$c_k-2\geq a_k=\sum_{x_{k'}\in N_G(x_k)}a_{k'}=\sum_{x_{k'}\in N_G(x_k)}c_{k'}.$$This contradicts our assuption in the first paragraph of the proof.   
\end{proof}

Since $C_8$ does not enjoy the strong exchange property, in Lemma \ref{trianind} the assumption on independence number at most $3$ cannot be dropped.  On the other hand, Example \ref{Ex_triangle} below shows that being triangle-free cannot be dropped. 

\begin{Example}
\label{Ex_triangle}
Let $G$ be the finite graph on $V(G)=\{x_1, \ldots, x_6\}$ with$$E(G)=\{\{x_1, x_2\}, \{x_1, x_3\}, \{x_2, x_3\}, \{x_3, x_4\}, \{x_4, x_5\}, \{x_4, x_6\}\}.$$Then $G$ has a triangle and its independence number is $3$.  Let $\mathfrak{c}=(1,\ldots,1)\in \ZZ_{>0}^6$.  One has $\delta_{\mathfrak{c}}(I(G))=2$ and $$(x_1x_3)(x_4x_5), (x_2x_3)(x_4x_6)\in \Wc(\mathfrak{c},G), \, \, \, \, \, x_3x_4x_5x_6\notin \Wc(\mathfrak{c},G).$$ 
Hence $\Wc(\mathfrak{c},G)$ cannot enjoy the strong exchange property.  
\end{Example}

Finally, we can classify the cycles enjoying the strong exchange property.

\begin{Theorem} \label{maincyc}
The cycle $C_n$ with $n \geq 3$ satisfies the strong exchange property if and only if $3 \leq n \leq 7$.
\end{Theorem}

\begin{proof}
Since $C_3$ is a complete graph, Theorem \ref{completeSEP} implies that it enjoys the strong exchange property.  It follows from Lemma \ref{trianind} that each of $C_4, C_5, C_6$ and $C_7$ enjoy the strong exchange property. On the other hand, Lemma \ref{cycnot} guarantees that $C_n$ with $n \geq 8$ does not enjoy the strong exchange property.
\end{proof} 

\begin{Corollary}
The cycle $C_n$ with $n \geq 3$ is of Veronese type if and only if $3 \leq n \leq 7$.
\end{Corollary}

\section{Trees}
We classify the trees enjoying the strong exchange property.  First of all, we classify the paths enjoying the strong exchange property.
Let $P_n$ be the path of length $n-1$ on $V(P_n) = \{x_1, \ldots, x_n\}$ with $E(G) = \{\{x_1,x_2\}, \ldots, \{x_{n-1},x_n\}\}$.  It follows from Lemma \ref{trianind} that $P_n$ enjoys the strong exchange property if $2 \leq n \leq 6$.  

\begin{Lemma} \label{pathnot}
The path $P_n$ with $n \geq 7$ does not enjoy the strong exchange property.
\end{Lemma}

\begin{proof}
Let $\mathfrak{c}=(c_1, \ldots, c_n)\in \ZZ_{>0}^n$ be defined by
$$c_i=
\left\{
	\begin{array}{ll}
		2  &  i=3,7,\\
		1 & {\rm otherwise}
	\end{array}
\right.$$
Set $\delta:=\delta_{\mathfrak{c}}(I(P_n))$.

\medskip

{\bf Claim.} $\delta=\lfloor \frac{n}{2}\rfloor$.

\medskip

\noindent
{\it Proof of Claim.} First, assume that $n$ is even. Then $P_n$ is a bipartite graph on the vertex set $X_1\sqcup X_2$ with $X_1=\{x_1,x_3, \ldots, x_{n-1}\}$ and $X_2=\{x_2,x_4, \ldots, x_n\}$. Since$$\sum_{x_i\in X_2}c_i=\frac{n}{2},$$it follows that $\delta\leq \frac{n}{2}$. On the other hand,$$(x_1x_2)(x_3x_4)\cdots (x_{n-1}x_n)$$is a $\mathfrak{c}$-bounded monomial of degree $n$. Thus, $\delta=\frac{n}{2}$.

Next, assume that $n$ is odd. The monomial$$(x_2x_3)(x_3x_4)(x_5x_6)(x_7x_8)\cdots (x_{n-2}x_{n-1})$$shows that $\delta\geq \frac{n-1}{2}$. To prove the reverse inequality, note that $P_n$ is a bipartite graph on the vertex set $X_1\sqcup X_2$ with $X_1=\{x_1,x_3, \ldots, x_n\}$ and $X_2=\{x_2,x_4, \ldots, x_{n-1}\}$. Moreover,$$\sum_{x_i\in X_2}c_i=\frac{n-1}{2}.$$Thus, $\delta\leq \frac{n-1}{2}$. This completes the proof of Claim.

\medskip

Assume that $n$ is even. Consider two monomials$$w_1=(x_1x_2)(x_3x_4)(x_5x_6)\cdots (x_{n-1}x_n)$$and$$w_2=(x_2x_3)(x_3x_4)(x_6x_7)(x_7x_8)\cdots (x_{n-1}x_n)$$in $\Wc(\mathfrak{c},P_n)$. Obviously, ${\rm deg}_{x_3}(w_2)> {\rm deg}_{x_3}(w_1)$ and ${\rm deg}_{x_5}(w_1)>{\rm deg}_{x_5}(w_2)$. If $\Wc(\mathfrak{c},P_n)$ enjoys the strong exchange property, then the monomial $x_3w_1/x_5$ must belong to $\Wc(\mathfrak{c},C_n)$ which is impossible, as this monomial is divisible by $x_1x_2x_3^2$.

Finally, assume that $n$ is odd. Then consider two monomials$$w_1=(x_1x_2)(x_3x_4)(x_5x_6)\cdots (x_{n-2}x_{n-1})$$and$$w_2=(x_2x_3)(x_3x_4)(x_6x_7)(x_7x_8)(x_9x_{10})\cdots (x_{n-2}x_{n-1})$$in $\Wc(\mathfrak{c},P_n)$. Obviously, ${\rm deg}_{x_3}(w_2)> {\rm deg}_{x_3}(w_1)$ and ${\rm deg}_{x_5}(w_1)>{\rm deg}_{x_5}(w_2)$. If $\Wc(\mathfrak{c},P_n)$ enjoys the strong exchange property, then the monomial $x_3w_1/x_5$ must belong to $\Wc(\mathfrak{c},P_n)$ which is impossible, as this monomial is divisible by $x_1x_2x_3^2$.
\end{proof}

\begin{Theorem} \label{stexchpath}
The path $P_n$ with $n \geq 2$ satisfies the strong exchange property if and only if $2 \leq n \leq 6$.
\end{Theorem}

\begin{Corollary}
The path $P_n$ with $n \geq 2$ is of Veronese type if and only if $2 \leq n \leq 6$.
\end{Corollary}

We now turn to a classification of the trees enjoying the strong exchange property.

\begin{Lemma} \label{treedel}
Let $G$ be a finite graph on $V(G) = \{x_1, \ldots, x_n\}$ which enjoys the strong exchange property and suppose that $x_n$ is a leaf of $G$. Then $G-x_n$ enjoys the strong exchange property.    
\end{Lemma}

\begin{proof}
Set $H:=G-x_n$ and suppose that $x_{n-1}$ is the unique neighbor of $x_n$. On the contrary, assume that $H$ does not enjoy the strong exchange property.  Choose $\mathfrak{c'}=(c_1', \ldots, c_{n-1}')\in \ZZ_{>0}^{n-1}$ for which  $\Wc(\mathfrak{c'},H)$ does not enjoy the strong exchange property. Thus, there are two monomials $w_i,w_j\in \Wc(\mathfrak{c'},H)$ and two variables $x_{\xi},x_{\rho}$ such that 
${\rm deg}_{x_{\xi}}(w_i)> {\rm deg}_{x_{\xi}}(w_j)$ and ${\rm deg}_{x_{\rho}}(w_i)< {\rm deg}_{x_{\rho}}(w_j)$, but $x_{\rho}(w_i/x_{\xi})\notin \Wc(\mathfrak{c'},H)$.  Define the vector $\mathfrak{c}=(c_1, \ldots, c_n)\in \ZZ_{>0}^n$ as follows.
$$c_i=
\left\{
	\begin{array}{ll}
		c_i'  &  1\leq i\leq n-2,\\
        c_{n-1}'+1  &  i=n-1,\\
		  1 & i=n
	\end{array}
\right.$$
We show that $\Wc(\mathfrak{c},G)$ does not enjoy the strong exchange property. 

Set $\delta:=\delta_{\mathfrak{c}}(I(G))$ and $\delta':=\delta_{\mathfrak{c'}}(I(H))$. As $x_{n-1}$ is the unique neighbor of $x_n$, one has $\delta=\delta'+1$. Then $u_i:=(x_{n-1}x_n)w_i$ and $u_j:=(x_{n-1}x_n)w_j$ belong to $\Wc(\mathfrak{c},G)$. Moreover, ${\rm deg}_{x_{\xi}}(u_i)> {\rm deg}_{x_{\xi}}(u_j)$ and ${\rm deg}_{x_{\rho}}(u_i)< {\rm deg}_{x_{\rho}}(u_j)$. However, since $x_{\rho}(w_i/x_{\xi})\notin \Wc(\mathfrak{c'},H)$, one has$$x_{\rho}(u_i/x_{\xi})=(x_{n-1}x_n)x_{\rho}(w_i/x_{\xi})\notin \Wc(\mathfrak{c},G),$$
a contradiction. Hence, $H$ enjoys the strong exchange property, as desired.
\end{proof}

\begin{Corollary} \label{p7forb}
Every tree enjoying the strong exchange property is $P_7$-free.    
\end{Corollary}

\begin{proof}
It follows from Theorem \ref{stexchpath} that $P_7$ does not enjoy the strong exchange property. The assertion now  follows by repeated applications of  Lemma \ref{treedel}.  
\end{proof}

\begin{Lemma} \label{pathleaf}
The finite graph obtained from $P_n$ with $n \geq 2$ by attaching two pendant edges to each of its endpoints $x_1$ and $x_n$ does not enjoy the strong exchange property.  
\end{Lemma}

\begin{proof}
Let $G$ be the finite graph obtained from $P_n$ by attaching two pendant edges to each of its endpoints.  Let $V(G)=\{x_1, \ldots, x_{n+4}\}$ and  
 $$E(G)=
 E(P_n)
 \cup\{\{x_1, x_{n+1}\}, \{x_1, x_{n+2}\}, \{x_n, x_{n+3}\}, \{x_n, x_{n+4}\}\}.$$
Consider the vector $\mathfrak{c}=(1, \ldots, 1)\in \ZZ_{>0}^{n+4}$. We show that $\Wc(\mathfrak{c},G)$ does not enjoy the strong exchange property. Set $\delta:=\delta_{\mathfrak{c}}(IG))$.

First, assume that $n$ is even. Then $\delta=(n+2)/2$.  The monomials 
$$w_1=(x_1x_{n+1})(x_2x_3)\cdots (x_{n-2}x_{n-1})(x_{n}x_{n+3})$$
and
$$w_2=(x_1x_{n+2})(x_2x_3)\cdots (x_{n-2}x_{n-1})(x_{n}x_{n+4})$$belong to $\Wc(\mathfrak{c},G)$ with $${\rm deg}_{x_{n+3}}(w_1)> {\rm deg}_{x_{n+3}}(w_2), \, \, \, \, \, \, \, \, \, \, {\rm deg}_{x_{n+2}}(w_1)< {\rm deg}_{x_{n+2}}(w_2).$$
If $\Wc(\mathfrak{c},G)$ enjoys the strong exchange property, then  $x_{n+2}w_1/x_{n+3} \in \Wc(\mathfrak{c},G)$, which is impossible, as this monomial is divisible by $x_1x_{n+1}x_{n+2}$.

Second, assume that $n$ is odd. Then $\delta=(n+1)/2$.  Considering 
$$w_1=(x_1x_{n+1})(x_2x_3)\cdots (x_{n-3}x_{n-2})(x_{n}x_{n+3})$$
and
$$w_2=(x_1x_{n+2})(x_2x_3)\cdots (x_{n-3}x_{n-2})(x_{n}x_{n+4})$$belonging to $\Wc(\mathfrak{c},G)$, the same argument as above shows that $\Wc(\mathfrak{c},G)$ does not enjoy the strong exchange property.
\end{proof}

\begin{Corollary} \label{twodegthree}
Every tree having two distinct vertices of degree at least three does not enjoy the strong exchange property.  
\end{Corollary}

\begin{proof}
Let $G$ be a tree on $V(G) = \{x_1, \ldots, x_n\} $ and suppose that ${\rm deg}_G(x_i)\geq 3$ and ${\rm deg}_G(x_j)\geq 3$, $i \neq j$.  Let $P$ denote the unique path of $G$ which connects $x_i$ and $x_j$. By assumption, there are vertices $x_p, x_q, x_{p'}, x_{q'}\not\in V(P)$ for which  $$\{x_i,x_p\}, \{x_i,x_q\}, \{x_j,x_{p'}\}, \{x_j,x_{q'}\}\in E(G).$$  Let $H$ denote the induced subgraph of $G$ on $V(P)\cup\{x_p, x_q, x_{p'}, x_{q'}\}$.  If $G$ enjoys the strong exchange property, then repeated applications of Lemma \ref{treedel} guarantee that $H$ enjoys the strong exchange property. This contradicts Lemma \ref{pathleaf}.
\end{proof}

\begin{Lemma} \label{specialsub}
Let $G$ be a tree which has an induced subgraph $H$ on the vertex set $V(H)=\{x_1, \ldots x_6\}$ with the edge set$$E(H)=\{\{x_1,x_2\}, \{x_2,x_3\}, \{x_2,x_4\}, \{x_4,x_5\}, \{x_5,x_6\}\}.$$Then $G$ does not enjoy the strong exchange property. 
\end{Lemma}

\begin{proof}
On the contrary, suppose that $G$ enjoys the strong exchange property. Then repeated applications of Lemma \ref{treedel} say that $H$ enjoys the strong exchange property. We show that this is not the case.

Let $\mathfrak{c}=(1,1,1,1,1,1)\in\ZZ_{>0}^6$. Then $\delta_{\mathfrak{c}}(I(H))=2$ and  $w_1=(x_1x_2)(x_5x_6)$ and 
$w_2=(x_2x_3)(x_4x_5)$ belong to $\Wc(\mathfrak{c},H)$ with ${\rm deg}_{x_6}(w_1)> {\rm deg}_{x_6}(w_2)$ and ${\rm deg}_{x_3}(w_1)< {\rm deg}_{x_3}(w_2).$ If $\Wc(\mathfrak{c},H)$ enjoys the strong exchange property, then $x_3w_1/x_6 \in \Wc(\mathfrak{c},G)$, which is impossible, as this monomial is divisible by $x_1x_2x_3$. 
\end{proof}

\begin{Lemma} \label{starwhisker}
A finite graph $G$ which is obtained from a star graph by attaching at most one pendant edge to each of its leaves enjoys the strong exchange property.   
\end{Lemma}

\begin{proof}
Let $0 \leq k \leq n$ be two integers.  Assume that $$V(G)=\{x_0, x_1,\ldots, x_n, x_{n+1}, \ldots, x_{n+k}\}$$and
$$E(G)=\{\{x_0,x_i\} :  1\leq i\leq n\}\cup\{\{x_i,x_{n+i}\} : 1\leq i\leq k\}.$$
Let $\mathfrak{c}=(c_0\ldots, c_{n+k}) \in \ZZ_{>0}^{n+k+1}$. We show that $\Wc(\mathfrak{c},G)$ enjoys the strong exchange property. Set $\delta:=\delta_{c}(I(G))$. If there is $1\leq t\leq k$ with $c_{n+t}> c_t$, then for the vector $\mathfrak{c'}$ obtained from $\mathfrak{c}$ by replacing $c_{t+n}$ with $c_t$, one has $\Wc(\mathfrak{c},G)=\Wc(\mathfrak{c'},G)$. Hence, from the beginning we assume that $c_{n+t}\leq c_t$ for each $1\leq t\leq k$.  

\medskip

{\bf (Case 1)} Suppose that there are $v\in \Wc(\mathfrak{c},G)$ and $1\leq i\leq n$ with ${\rm deg}_{x_i}(v) < c_i$.  

\medskip

{\bf Claim 1.} ${\rm deg}_{x_0}(v)=c_0$ and  ${\rm deg}_{x_{n+j}}(v)=c_{n+j}$, for each $1\leq j \leq k$.

\medskip

\noindent
{\it Proof of Claim 1.}  If ${\rm deg}_{x_0}(v) < c_0$, then $(x_0x_i)v\in (I(G)^{\delta +1})_{\mathfrak{c}}$, a contradiction, which  shows that ${\rm deg}_{x_0}(v)=c_0$.

Suppose that there is $1\leq j\leq k$ with ${\rm deg}_{x_{n+j}}(v) < c_{n+j}$. If ${\rm deg}_{x_j}(v) < c_j$ then $(x_{j}x_{n+j})v\in (I(G)^{\delta +1})_{\mathfrak{c}}$, a contradiction. Hence, ${\rm deg}_{x_j}(v)=c_j$. Thus, in the representation of $v=e_1\cdots e_{\delta}$, there is an edge which is incident to $x_j$. If all of such edges are incident  to $x_{n+j}$, then$$c_j={\rm deg}_{x_j}(v)={\rm deg}_{x_{n+j}}(v)<c_{n+j},$$which contradicts our assumption from the first paragraph of the proof. Therefore, in the representation of $v=e_1\cdots e_{\delta}$, there is an edge, say, $e_1$ which is equal to $\{x_0,x_j\}$. This implies that$$(x_ix_{n+j})v=(x_0x_i)(x_jx_{n+j})e_2\cdots e_{\delta}\in (I(G)^{\delta+1})_{\mathfrak{c}},$$a contradiction. This completes the proof of Claim 1.

\medskip

It follows from Claim 1 that$$\delta={\rm deg}_{x_0}(v)+\sum_{j=1}^k{\rm deg}_{x_{n+j}}(v)=c_0+\sum_{j=1}^kc_{n+j}.$$Let $u$ be a monomial in $\Wc(\mathfrak{c},G)$. It follows from the expression of $\delta$ as above, that $u$ is divisible by $x_0^{c_0}x_{n+1}^{c_{n+1}}\cdots x_{n+k}^{c_{n+k}}$. Thus, the structure of $G$ implies that $$u=(x_1x_{n+1})^{c_{n+1}}\cdots (x_kx_{n+k})^{c_{n+k}}u',$$where $u'$ can be an arbitrary monomial in $\Wc(\mathfrak{c''},H)$. Here, $H=G-\{x_{n+1}, \ldots, x_{n+k}\}$ is a star graph and $\mathfrak{c''}=(c_0'', \ldots, c_{n''})\in \ZZ_{>0}^{n+1}$ is defined by
$$c_i''=
\left\{
	\begin{array}{ll}
		c_0  &  i=0,\\
        c_i-c_{n+i}  &  1\leq i\leq k,\\
		  c_i & k+1\leq i\leq n.
	\end{array}
\right.$$
We know from Theorem \ref{completeSEP} that $\Wc(\mathfrak{c''},H)$ enjoys the strong exchange property. Hence, $\Wc(\mathfrak{c},G)$ enjoys the same property.

\medskip

{\bf (Case 2)} Suppose that for each $v\in \Wc(\mathfrak{c},G)$ and for each $1\leq i\leq n$, one has ${\rm deg}_{x_i}(v)=c_i$. This implies that$$\delta=\sum_{i=1}^n{\rm deg}_{x_i}(v)=\sum_{i=1}^nc_i.$$Therefore, each monomial $u\in\Wc(\mathfrak{c},G)$ is of the form $u'x_1^{c_1}\cdots x_n^{c_n}$, where $u'$ is a monomial in variables $x_0, x_{n+1}, \ldots, x_{n+k}$ with ${\rm deg}(u')=\sum_{i=1}^nc_i$. Moreover, for each $x_j\in \{x_0, x_{n+1}, \ldots, x_{n+k}\}$, one has ${\rm deg}_{x_j}(u')\leq c_j$.

\medskip

{\bf Claim 2.} Let $w$ be a monomial in $x_0, x_{n+1}, \ldots, x_{n+k}$ with ${\rm deg}(w)=\sum_{i=1}^nc_i$ satisfying ${\rm deg}_{x_j}(w)\leq c_j$ for each $x_0, x_{n+1}, \ldots, x_{n+k}$. Then $wx_1^{c_1}\cdots x_n^{c_n}\in \Wc(\mathfrak{c},G)$.

\medskip

\noindent
{\it Proof of Claim 2.}  By our assumption from the first paragraph of the proof, we know that $c_{n+t}\leq c_t$, for each $1 \leq t \leq k$. Let $w=x_0^{a_0}x_{n+1}^{a_{n+1}}\cdots x_{n+k}^{a_{n+k}}$. Then for each $1 \leq t \leq k$, one has $a_{n+t}\leq c_{n+t}\leq c_t$. Thus,$$wx_1^{c_1}\cdots x_n^{c_n}=(x_1x_{n+1})^{a_{n+1}}\cdots (x_kx_{n+k})^{a_{n+k}}w',$$where $w'=x_0^{a_0}x_1^{c_1-a_{n+1}}\cdots x_k^{c_k-a_{n+k}}x_{k+1}^{c_{k+1}}\cdots x_n^{c_n}$. Recall that$$a_0+a_{n+1}+ \cdots a_{n+k}={\rm deg}(w)=\sum_{i=1}^nc_i.$$This yields that$${\rm deg}(w')={\rm deg}(w)+\sum_{i=1}^nc_i-2\sum_{j=1}^ka_{n+j}=2a_0.$$
Since ${\rm deg}_{x_0}(w')=a_0$ and since $x_0$ is adjacent to each of $x_1, \ldots, x_n$, we deduce that $w'$ can be written as the product of $a_0$ edges of $G$. Therefore, $wx_1^{c_1}\cdots x_n^{c_n}$ is the product of $a_0+a_{n+1}+\ldots+a_{n+k}={\rm deg}(w)=\delta$ edges of $G$. In other words, $wx_1^{c_1}\cdots x_n^{c_n}\in \Wc(\mathfrak{c},G)$, which completes the proof of Claim 2.

\medskip

It follows from Claim 2 and the argument before it that each $u\in\Wc(\mathfrak{c},G)$ is of the form $u'x_1^{c_1}\cdots x_n^{c_n}$, where $u'$ is an arbitrary monomial belonging to the minimal system of monomial generators of the algebra of Veronese type $A(d;\mathfrak{a})$, where  $d=\sum_{i=1}^nc_i$ and $\mathfrak{a}=(c_0, c_{n+1}, \ldots, c_{n+k})\in \ZZ_{>0}^{k+1}$. Thus, the required result follows.
\end{proof}

Finally, we can classify the trees enjoying the strong exchange property.

\begin{Theorem} \label{treeclass}
A tree $G$ enjoys the strong exchange property if and only if one of the followings holds.
\begin{itemize}
    \item [(i)] $G=P_6$; 
    \item [(ii)] $G$ is obtained from a star graph by attaching at most one pendant edge to each of its leaves.
\end{itemize}

\begin{proof}
The ``if'' part follows from Theorem \ref{stexchpath} and Lemma \ref{starwhisker}.  We prove the ``only if'' part.  Lemma  \ref{p7forb} says that $G$ is a $P_7$-free graph. Let $\ell$ denote the length of the longest path of $G$. One has $1\leq \ell\leq 5$.

\begin{itemize}
 \item If $\ell=1$, then $G=K_2$ is a tree as described in (ii).
 \item If $\ell=2$, then $G$ is a star graph, so a tree as described in (ii).

\item Suppose $\ell=3$. Let $P$ be a path of length three in $G$ with $V(P)=\{x_1, x_2, x_3, x_4\}$ and $E(P)=\{\{x_1, x_2\}, \{x_2, x_3\}, \{x_3, x_4\}\}.$ If $G=P$, then $G$ is a tree as described in (ii). Let $G\neq P$. Let, say, $x_5\in V(G)\setminus P$ which is adjacent to one vertex of $P$. Since $P$ is a maximal path of $G$, $\{x_1,x_5\}, \{x_4,x_5\}\notin E(G)$. Consequently, $x_5$ is adjacent to exactly one of $x_2, x_4$. Let $\{x_2,x_5\}\in E(G)$ by symmetry.   Corollary \ref{twodegthree} says that there is no vertex $x_i\in V(G)\setminus V(P)$ with $\{x_3,x_i\}\in E(G)$. Since $G$ has no path of length $4$, each $x_j\in V(G)\setminus V(P)$ with $\{x_2,x_i\}\in E(G)$ is a leaf of $G$. Thus, $G$ is a tree as described in (ii).

\item Suppose $\ell=4$. Let $P$ be a path of length four in $G$ and assume that $V(P)=\{x_1, x_2, x_3, x_4, x_5\}$ and $E(P)=\{\{x_1, x_2\}, \{x_2, x_3\}, \{x_3, x_4\}, \{x_4, x_5\}\}.$
If $G=P$, then $G$ is a tree as described in (ii). Let $G\neq P$. Let, say, $x_6\in V(G)\setminus P$ which is adjacent to one vertex of $P$. Since $\ell=4$, one has $\{x_1,x_6\}, \{x_5,x_6\}\notin E(G)$.  On the other hand, Lemma \ref{specialsub} says that $\{x_2, x_6\}, \{x_4, x_6\}\notin E(G)$. Consequently, $x_6$ is adjacent to $x_3$. Furthermore, Corollary \ref{twodegthree} implies that every vertex in $V(G)\setminus V(P)$ has degree at most two in $G$. Thus, a similar argument as in the case $\ell=3$ shows that $G$ is a tree as described in (ii). 

\item Suppose $\ell=5$. Let $P$ be a path of length five in $G$. Then a similar argument as in the case $\ell=4$ based on Lemma \ref{specialsub} shows that no vertex in $V(G)\setminus V(P)$ can be adjacent to a vertex of $P$. Thus $G=P=P_6$.
\end{itemize}
Now, the proof of ``only if'' part is complete.
\end{proof}
\end{Theorem}

\section{Unicyclic graphs}
  In the present section, we classify unicycle graphs enjoying the strong exchange property.  Our classification is summarized in Theorem \ref{classification_unicyclic}. 

In this section, when we consider a cycle $C_n$ on the vertex set $\{x_1, \ldots, x_n\}$, we always mean that $E(C_n) = \{\{x_1,x_2\}, \{x_2,x_3\}, \ldots, \{x_{n-1},x_n\},\{x_n, x_1\}\}$. 

\begin{Lemma} \label{c7pend}
The finite graph $G$ obtained from the cycle $C_7$ by attaching a pendant edge to one of its vertices does not enjoy the strong exchange property.    
\end{Lemma}

\begin{proof}
Let $V(G)=\{x_1, \ldots, x_8\}$ and  
 $E(G)= E(C_7)
 \cup\{\{x_1, x_8\}\}$, where $V(C_7)=\{x_1, \ldots, x_7\}$.
Let $\mathfrak{c}=(2,3,1,1,2,1,1,2)\in \ZZ_{>0}^8$. Then $\delta_{\mathfrak{c}}(I(G))=5$. The monomials $$w_1=(x_1x_8)^2(x_2x_3)(x_4x_5)(x_5x_6), \, \, \, \, \,  \, \, \, \, \, w_2=(x_1x_2)^2(x_2x_3)(x_4x_5)(x_6x_7)$$ belong to $\Wc(\mathfrak{c},G)$ with ${\rm deg}_{x_5}(w_1)> {\rm deg}_{x_5}(w_2)$ and ${\rm deg}_{x_2}(w_1)< {\rm deg}_{x_2}(w_2)$.  If $\Wc(\mathfrak{c},G)$ enjoys the strong exchange property, then  $x_2w_1/x_5 \in \Wc(\mathfrak{c},G)$, which is impossible, as this monomial is divisible by $x_1^2x_2^2x_8^2$.
\end{proof}

\begin{Lemma} \label{c6pend}
The finite graph $G$ obtained from the cycle $C_6$ by attaching a pendant edge to one of its vertices does not enjoy the strong exchange property.    
\end{Lemma}

\begin{proof}
Let $V(G)=\{x_1, \ldots, x_7\}$ and  
 $E(G)= E(C_6)
 \cup\{\{x_1, x_7\}\}$, where $V(C_6)=\{x_1, \ldots, x_6\}$.
Let $\mathfrak{c}=(1,2,1,1,1,1,1)\in \ZZ_{>0}^7$. Then $\delta_{\mathfrak{c}}(I(G))=3$. The monomials $$w_1=(x_1x_7)(x_2x_3)(x_4x_5), \, \, \, \, \, \, \, \, \, \, w_2=(x_1x_2)(x_2x_3)(x_5x_6)$$ belong to $\Wc(\mathfrak{c},G)$ with ${\rm deg}_{x_4}(w_1)> {\rm deg}_{x_4}(w_2)$ and ${\rm deg}_{x_2}(w_1)< {\rm deg}_{x_2}(w_2)$.  If $\Wc(\mathfrak{c},G)$ enjoys the strong exchange property, then  $x_2w_1/x_4 \in \Wc(\mathfrak{c},G)$, which is impossible, as this monomial is divisible by $x_1x_2^2x_7$.
\end{proof}

\begin{Lemma} \label{c5pendad}
The finite graph $G$ obtained from the cycle $C_5$ by attaching a pendant edge to two adjacent vertices of $C_5$ does not enjoy the strong exchange property.    
\end{Lemma}

\begin{proof}
Let $V(G)=\{x_1, \ldots, x_7\}$ and  
 $E(G)= E(C_5)
 \cup\{\{x_1, x_6\}, \{x_5,x_7\}\}$,  where $V(C_5)=\{x_1, \ldots, x_5\}$.
Let $\mathfrak{c}=(1,2,1,1,1,1,1)\in \ZZ_{>0}^7$. Then $\delta_{\mathfrak{c}}(I(G))=3$. The monomials $$w_1=(x_1x_6)(x_2x_3)(x_5x_7), \, \, \, \, \, \, \, \, \, \,  w_2=(x_1x_2)(x_2x_3)(x_4x_5)$$ belong to $\Wc(\mathfrak{c},G)$ with ${\rm deg}_{x_7}(w_1)> {\rm deg}_{x_7}(w_2)$ and ${\rm deg}_{x_2}(w_1)< {\rm deg}_{x_2}(w_2)$.  If $\Wc(\mathfrak{c},G)$ enjoys the strong exchange property, then  $x_2w_1/x_7 \in \Wc(\mathfrak{c},G)$, which is impossible, as this monomial is divisible by $x_1x_2^2x_6$.
\end{proof}

\begin{Lemma} \label{c5pendnoad}
The finite graph $G$ obtained from the cycle $C_5$ by attaching a pendant edge to two non-adjacent vertices of $C_5$ does not enjoy the strong exchange property.    
\end{Lemma}

\begin{proof}
Let $V(G)=\{x_1, \ldots, x_7\}$ and  
 $E(G)= E(C_5)
 \cup\{\{x_1, x_6\}, \{x_4,x_7\}\}$, where $V(C_5)=\{x_1, \ldots, x_5\}$.
Let $\mathfrak{c}=(1,2,1,1,1,1,1)\in \ZZ_{>0}^7$. Then $\delta_{\mathfrak{c}}(I(G))=3$.  The monomials $$w_1=(x_1x_6)(x_2x_3)(x_4x_7), \, \, \, \, \, \, \, \, \, \, w_2=(x_1x_2)(x_2x_3)(x_4x_5)$$ belong to $\Wc(\mathfrak{c},G)$ with ${\rm deg}_{x_7}(w_1)> {\rm deg}_{x_7}(w_2)$ and ${\rm deg}_{x_2}(w_1)< {\rm deg}_{x_2}(w_2)$.  If $\Wc(\mathfrak{c},G)$ enjoys the strong exchange property, then  $x_2w_1/x_7 \in \Wc(\mathfrak{c},G)$, which is impossible, as this monomial is divisible by $x_1x_2^2x_6$.
\end{proof}

\begin{Lemma} \label{c5twopend}
The finite graph $G$ obtained from the cycle $C_5$ by attaching two pendant edges to one of its vertices does not enjoy the strong exchange property.    
\end{Lemma}

\begin{proof}
Let $V(G)=\{x_1, \ldots, x_7\}$ and  
 $E(G)= E(C_5)
 \cup\{\{x_1, x_6\}, \{x_1,x_7\}\}$, where $V(C_5)=\{x_1, \ldots, x_5\}$.
Let $\mathfrak{c}=(1,1,2,1,1,1,1)\in \ZZ_{>0}^7$. Then $\delta_{\mathfrak{c}}(I(G))=3$. The monomials $$w_1=(x_1x_6)(x_2x_3)(x_3x_4), \, \, \, \, \, \, \, \, \, \, w_2=(x_1x_7)(x_2x_3)(x_4x_5)$$ belong to $\Wc(\mathfrak{c},G)$ with ${\rm deg}_{x_3}(w_1)> {\rm deg}_{x_3}(w_2)$ and ${\rm deg}_{x_7}(w_1)< {\rm deg}_{x_7}(w_2)$.  If $\Wc(\mathfrak{c},G)$ enjoys the strong exchange property, then  $x_7w_1/x_3 \in \Wc(\mathfrak{c},G)$, which is impossible, as this monomial is divisible by $x_1x_6x_7$.
\end{proof}

\begin{Lemma} \label{c5path}
The finite graph $G$ obtained from the cycle $C_5$ by attaching a path of length three to one of its vertices does not enjoy the strong exchange property.    
\end{Lemma}

\begin{proof}
Let $V(G)=\{x_1, \ldots, x_8\}$ and  
 $E(G)= E(C_5)
 \cup\{\{x_1, x_6\}, \{x_6,x_7\}, \{x_7,x_8\}\}$, where we consider $C_5$ on the vertices $x_1, \ldots, x_5$.
Let $\mathfrak{c}=(1,1,1,2,1,2,1,1)\in \ZZ_{>0}^8$. Then $\delta_{\mathfrak{c}}(I(G))=4$. The monomials $$w_1=(x_1x_6)(x_3x_4)(x_4x_5)(x_7x_8), \, \, \, \, \, \, \, \, \, \, w_2=(x_1x_6)(x_2x_3)(x_4x_5)(x_6x_7)$$ belong to $\Wc(\mathfrak{c},G)$ with ${\rm deg}_{x_4}(w_1)> {\rm deg}_{x_4}(w_2)$ and ${\rm deg}_{x_6}(w_1)< {\rm deg}_{x_6}(w_2)$.  If $\Wc(\mathfrak{c},G)$ enjoys the strong exchange property, then  $x_6w_1/x_4 \in \Wc(\mathfrak{c},G)$, which is impossible, as this monomial is divisible by $x_6^2x_7x_8$.
\end{proof}

\begin{Lemma} \label{c5star}
The finite graph $G$ on the vertex set $V(G)=\{x_1, \ldots, x_8\}$ with the edge set$$E(G)= E(C_5)
 \cup\{\{x_1, x_6\}, \{x_6,x_7\}, \{x_6,x_8\}\},$$where $V(C_5)=\{x_1, \ldots,x_5\}$, does not enjoy the strong exchange property.    
\end{Lemma}

\begin{proof}
Let $\mathfrak{c}=(1,1,1,1,1,1,1,1)\in \ZZ_{>0}^8$. Then $\delta_{\mathfrak{c}}(I(G))=3$. The monomials $$w_1=(x_1x_2)(x_3x_4)(x_6x_7), \, \, \, \, \, \, \, \, \, \, w_2=(x_1x_2)(x_4x_5)(x_6x_8)$$ belong to $\Wc(\mathfrak{c},G)$ with ${\rm deg}_{x_3}(w_1)> {\rm deg}_{x_3}(w_2)$ and ${\rm deg}_{x_8}(w_1)< {\rm deg}_{x_8}(w_2)$.  If $\Wc(\mathfrak{c},G)$ enjoys the strong exchange property, then  $x_8w_1/x_3 \in \Wc(\mathfrak{c},G)$, which is impossible, as this monomial is divisible by $x_6x_7x_8$.
\end{proof}

\begin{Lemma} \label{c4twopend}
The finite graph $G$ obtained from the cycle $C_4$ by attaching two pendant edges to one of its vertices does not enjoy the strong exchange property.    
\end{Lemma}

\begin{proof}
Let $V(G)=\{x_1, \ldots, x_6\}$ and  
 $E(G)= E(C_4)
 \cup\{\{x_1, x_5\}, \{x_1,x_6\}\}$, where $V(C_5)=\{x_1, \ldots, x_4\}$.
Let $\mathfrak{c}=(1,1,1,1,1,1)\in \ZZ_{>0}^6$. Then $\delta_{\mathfrak{c}}(I(G))=2$. The monomials $$w_1=(x_1x_5)(x_2x_3), \, \, \, \, \, \, \, \, \, \, w_2=(x_1x_6)(x_3x_4)$$ belong to $\Wc(\mathfrak{c},G)$ with ${\rm deg}_{x_2}(w_1)> {\rm deg}_{x_2}(w_2)$ and ${\rm deg}_{x_6}(w_1)< {\rm deg}_{x_6}(w_2)$.  If $\Wc(\mathfrak{c},G)$ enjoys the strong exchange property, then  $x_6w_1/x_2 \in \Wc(\mathfrak{c},G)$, which is impossible, as this monomial is divisible by $x_1x_5x_6$.
\end{proof}

\begin{Lemma} \label{c4pendpath}
The finite graph $G$ on the vertex set $\{x_1, \ldots, x_7\}$ with the edge set$$E(G)=E(C_4)\cup\{\{x_1,x_5\},\{x_5,x_6\},\{x_4,x_7\}\},$$where $V(C_4)=\{x_1,x_2,x_3,x_4\}$, does not enjoy the strong exchange property.    
\end{Lemma}

\begin{proof}
Let $\mathfrak{c}=(1,1,2,1,1,1,1)\in \ZZ_{>0}^7$. Then $\delta_{\mathfrak{c}}(I(G))=3$.  The monomials $$w_1=(x_2x_3)(x_4x_7)(x_5x_6), \, \, \, \, \, \, \, \, \, \, w_2=(x_1x_5)(x_2x_3)(x_3x_4)$$ belong to $\Wc(\mathfrak{c},G)$ with ${\rm deg}_{x_6}(w_1)> {\rm deg}_{x_6}(w_2)$ and ${\rm deg}_{x_3}(w_1)< {\rm deg}_{x_3}(w_2)$. If $\Wc(\mathfrak{c},G)$ enjoys the strong exchange property, then  $x_3w_1/x_6 \in \Wc(\mathfrak{c},G)$, which is impossible, as this monomial is divisible by $x_3^2x_4x_7$.
\end{proof}

\begin{Lemma} \label{c4twopath}
The finite graph $G$ on the vertex set $\{x_1, \ldots, x_8\}$ with the edge set$$E(G)=E(C_4)\cup\{\{x_1,x_5\},\{x_5,x_6\},\{x_3,x_7\},\{x_7,x_8\}\},$$where $V(C_4)=\{x_1,x_2,x_3,x_4\}$, does not enjoy the strong exchange property.    
\end{Lemma}

\begin{proof}
Let $\mathfrak{c}=(3,1,3,1,3,3,3,3)\in \ZZ_{>0}^8$. Then $\delta_{\mathfrak{c}}(I(G))=8$. The monomials $$w_1=(x_1x_2)(x_1x_4)(x_3x_7)^3(x_5x_6)^3, \, \, \, \, \, \, \, \, \, \, w_2=(x_1x_5)^3(x_2x_3)(x_3x_4)(x_7x_8)^3$$ belong to $\Wc(\mathfrak{c},G)$ with ${\rm deg}_{x_3}(w_1)> {\rm deg}_{x_3}(w_2)$ and ${\rm deg}_{x_1}(w_1)< {\rm deg}_{x_1}(w_2)$. If $\Wc(\mathfrak{c},G)$ enjoys the strong exchange property, then  $x_1w_1/x_3 \in \Wc(\mathfrak{c},G)$, which is impossible, as this monomial is divisible by $x_1^3x_5^3x_6^3$.
\end{proof}

\begin{Lemma} \label{c4star}
The finite graph $G$ on the vertex set $\{x_1, \ldots, x_7\}$ with the edge set$$E(G)=E(C_4)\cup\{\{x_1,x_5\},\{x_5,x_6\},\{x_5,x_7\}\},$$ 
where $V(C_4)=\{x_1,x_2,x_3,x_4\}$,
does not enjoy the strong exchange property.    
\end{Lemma}

\begin{proof}
Let $\mathfrak{c}=(1,2,1,1,1,1,1)\in \ZZ_{>0}^7$. Then $\delta_{\mathfrak{c}}(I(G))=3$. The monomials $$w_1=(x_1x_2)(x_2x_3)(x_5x_6), \, \, \, \, \, \, \, \, \, \, w_2=(x_1x_2)(x_3x_4)(x_5x_7)$$ belong to $\Wc(\mathfrak{c},G)$ with ${\rm deg}_{x_2}(w_1)> {\rm deg}_{x_2}(w_2)$ and ${\rm deg}_{x_7}(w_1)< {\rm deg}_{x_7}(w_2)$. If $\Wc(\mathfrak{c},G)$ enjoys the strong exchange property, then  $x_7w_1/x_2 \in \Wc(\mathfrak{c},G)$, which is impossible, as this monomial is divisible by $x_5x_6x_7$.
\end{proof}

\begin{Lemma} \label{c4tpathlong}
The finite graph $G$ on the vertex set $\{x_1, \ldots, x_7\}$ with the edge set$$E(G)=E(C_4)\cup\{\{x_1,x_5\},\{x_5,x_6\},\{x_6,x_7\}\},$$
where $V(C_4)=\{x_1,x_2,x_3,x_4\}$,
does not enjoy the strong exchange property.    
\end{Lemma}

\begin{proof}
Let $\mathfrak{c}=(1,1,1,1,2,1,1)\in \ZZ_{>0}^7$. Then $\delta_{\mathfrak{c}}(I(G))=3$. The monomials $$w_1=(x_1x_5)(x_3x_4)(x_6x_7), \, \, \, \, \, \, \, \, \, \, w_2=(x_1x_5)(x_2x_3)(x_5x_6)$$ belong to $\Wc(\mathfrak{c},G)$ with ${\rm deg}_{x_4}(w_1)> {\rm deg}_{x_4}(w_2)$ and ${\rm deg}_{x_5}(w_1)< {\rm deg}_{x_5}(w_2)$. If $\Wc(\mathfrak{c},G)$ enjoys the strong exchange property, then  $x_5w_1/x_4 \in \Wc(\mathfrak{c},G)$, which is impossible, as this monomial is divisible by $x_5^2x_6x_7$.
\end{proof}

\begin{Lemma} \label{c4pendall}
The finite graph $G$ obtained from the cycle $C_4$ by attaching a pendant edge to each of the vertices of $C_4$ is of Veronese type and, in particular, enjoys the strong exchange property.    
\end{Lemma}

\begin{proof}
Let $V(G)=\{x_1, \ldots, x_8\}$ and$$E(G)= E(C_4)
 \cup\{\{x_1, x_5\}, \{x_2,x_6\},\{x_3,x_7\},\{x_4,x_8\}\},$$
 where $V(C_4)=\{x_1,x_2,x_3,x_4\}$.
Let $\mathfrak{c}=(c_1, \ldots, c_8)\in \ZZ_{>0}^8$. We show that $\Wc(\mathfrak{c},G)$ is of Veronese type. If there is $1\leq t\leq4$ with $c_{t+4}> c_t$, then for the vector $\mathfrak{c'}$ obtained from $\mathfrak{c}$ by replacing $c_{t+4}$ with $c_t$, one has $\Wc(\mathfrak{c},G)=\Wc(\mathfrak{c'},G)$. Hence, from the beginning we assume that $c_{t+4}\leq c_t$ for each $t=1,2,3,4$.  
Set $\delta:=\delta_{\mathfrak{c}}(I(G))$. 

\medskip

{\bf (Case 1)} Suppose that there are $v\in \Wc(\mathfrak{c},G)$ and $1\leq i\leq 4$ with ${\rm deg}_{x_i}(v) < c_i$. By symmetry, we may assume that $i=1$. Thus, ${\rm deg}_{x_1}(v) < c_1$. Assume that $v=e_1\cdots e_{\delta}$, where $e_1, \ldots, e_{\delta}$ are edges of $G$. If ${\rm deg}_{x_5}(v) < c_5$, then $(x_1x_5)v$ belongs to $(I(G)^{\delta +1})_{\mathfrak{c}}$, a contradiction. Therefore, ${\rm deg}_{x_5}(v)=c_5$. Similarly, ${\rm deg}_{x_2}(v)=c_2$ and ${\rm deg}_{x_4}(v)=c_4$. Assume that ${\rm deg}_{x_7}(v)< c_7$. If in the representation of $v=e_1\cdots e_{\delta}$, there is an edge, say, $e_1$, which is equal to $\{x_2,x_3\}$, then$$(x_1x_7)v=(x_1x_2)(x_3x_7)e_2\cdots e_{\delta}\in (I(G)^{\delta +1})_{\mathfrak{c}},$$a contradiction. So, the edge $\{x_2,x_3\}$ does not appear in the representation of $v$. Similarly, the edge $\{x_3,x_4\}$ does not appear in the representation of $v$. Thus, in the representation of $v$ every edge incident to $x_3$ is the edge $\{x_3,x_7\}$. Recall that by our assumption $c_3\geq c_7$. Hence,$${\rm deg}_{x_3}(v)={\rm deg}_{x_7}(v)< c_7\leq c_3.$$This yields that $(x_3x_7)v\in (I(G)^{\delta +1})_{\mathfrak{c}}$, a contradiction. This argument shows that ${\rm deg}_{x_7}(v)=c_7$. Thus, we proved that ${\rm deg}_{x_5}(v)=c_5$,  ${\rm deg}_{x_2}(v)=c_2$, ${\rm deg}_{x_4}(v)=c_4$ and ${\rm deg}_{x_7}(v)=c_7$. These equalities imply that $\delta=c_2+c_4+c_5+c_7$ and $v$ can be written as $$v=v'x_2^{c_2}x_4^{c_4}x_5^{c_5}x_7^{c_7},$$where $v'$ is a $(c_1, c_3, c_6, c_8)$-bounded monomial of degree $c_2+c_4+c_5+c_7$ on variables $x_1, x_3, x_6, x_8$. Moreover, since the unique neighbor of $x_5$ (resp. $x_7$) is $x_1$ (resp. $x_3$), the above equality implies that ${\rm deg}_{x_1}(v')\geq c_5$ and ${\rm deg}_{x_3}(v')\geq c_7$. Therefore,  
$v$ can be written as$$v=v''(x_1x_5)^{c_5}(x_3x_7)^{c_7}x_2^{c_2}x_4^{c_4},$$where $v''$ is a $(c_1-c_5,c_3-c_7,c_6,c_8)$-bounded monomial of degree $c_2+c_4$ on variables $x_1, x_3, x_6, x_8$. Conversely, it is easy to see that for any monomial $w$ on $x_1, x_3, x_6, x_8$ which is a $(c_1-c_5,c_3-c_7,c_6,c_8)$-bounded monomial of degree $c_2+c_4$, the monomial $w(x_1x_5)^{c_5}(x_3x_7)^{c_7}x_2^{c_2}x_4^{c_4}$ belongs to $\Wc(\mathfrak{c},G)$. This shows that the toric ring generated by monomials in $\Wc(\mathfrak{c},G)$ is the algebra of Veronese type$$A(c_2+c_4;(c_1-c_5,c_3-c_7,c_6,c_8)).$$

\medskip

{\bf (Case 2)} Suppose that for each $v\in \Wc(\mathfrak{c},G)$ and for each $i=1,2,3,4$, one has ${\rm deg}_{x_i}(v)=c_i$. As above, let $v=e_1\cdots e_{\delta}$ be a monomial in $\Wc(\mathfrak{c},G)$, where $e_1, \ldots, e_{\delta}$ are edges of $G$. If ${\rm deg}_{x_{i+4}}(v)=c_{i+4}$, for each $i=1,2,3,4$, then $v=x_1^{c_1}\cdots x_8^{c_8}$. Thus, $\Wc(\mathfrak{c},G)$ is a singleton, and we are done. So, suppose that there is an integer $i$ with $1\leq i\leq 4$ such that ${\rm deg}_{x_{i+4}}(v)<c_{i+4}$. We may assume that $i=1$. In other words, ${\rm deg}_{x_5}(v)< c_5$. Since$${\rm deg}_{x_5}(v)< c_5\leq c_1={\rm deg}_{x_1}(v),$$in the representation of $v$, there is an edge, say $e_1$, which is incident to $x_1$ but not to $x_5$. Hence, either $e_1=\{x_1,x_2\}$ or $e_1=\{x_1,x_4\}$. First, assume that $e_1=\{x_1,x_2\}$ and consider the monomial $$v'=x_5v/x_2=(x_1x_5)e_2\cdots e_{\delta}\in \Wc(\mathfrak{c},G).$$Since ${\rm deg}_{x_2}(v')< c_2$, this contradicts our assumption in this case. Similarly, if $e_1=\{x_1,x_4\}$, we obtain a contradiction. This completes the proof.
\end{proof}

\begin{Lemma} \label{c4pathpendad}
The finite graph $G$ on the vertex set $\{x_1, \ldots, x_7\}$ with the edge set$$E(G)=E(C_4)\cup\{\{x_1,x_5\},\{x_5,x_6\},\{x_3,x_7\}\},$$
where $V(C_4)=\{x_1,x_2,x_3,x_4\}$, is of Veronese type and, in particular, enjoys the strong exchange property.  
\end{Lemma}

\begin{proof}
Fix $\mathfrak{c}=(c_1, \ldots, c_7)\in \ZZ_{>0}^7$. We show that $\Wc(\mathfrak{c},G)$ is of Veronese type. If $c_6> c_5$, then for the vector $\mathfrak{c'}$ obtained from $\mathfrak{c}$ by replacing $c_6$ with $c_5$, one has $\Wc(\mathfrak{c},G)=\Wc(\mathfrak{c'},G)$. Hence, from the beginning we assume that $c_6\leq c_5$. By a similar argument, we may also assume that $c_7\leq c_3$. Set $\delta:=\delta_{\mathfrak{c}}(I(G))$. 

\medskip

{\bf (Case 1)} Suppose that there is a monomial $v\in \Wc(\mathfrak{c},G)$ with ${\rm deg}_{x_1}(v) < c_1$. Assume that $v=e_1\cdots e_{\delta}$, where $e_1, \ldots, e_{\delta}$ are edges of $G$. If ${\rm deg}_{x_5}(v) < c_5$, then $(x_1x_5)v$ belongs to $(I(G)^{\delta +1})_{\mathfrak{c}}$, a contradiction. Therefore, ${\rm deg}_{x_5}(v)=c_5$. Similarly, ${\rm deg}_{x_2}(v)=c_2$ and ${\rm deg}_{x_4}(v)=c_4$. Assume that ${\rm deg}_{x_7}(v)< c_7$. If in the representation of $v=e_1\cdots e_{\delta}$, there is an edge, say, $e_1$, which is equal to $\{x_2,x_3\}$, then$$(x_1x_7)v=(x_1x_2)(x_3x_7)e_2\cdots e_{\delta}\in (I(G)^{\delta +1})_{\mathfrak{c}},$$a contradiction. So, the edge $\{x_2,x_3\}$ does not appear in the representation of $v$. Similarly, the edge $\{x_3,x_4\}$ does not appear in the representation of $v$. Thus, in the representation of $v$ every edge incident to $x_3$ is the edge $\{x_3,x_7\}$. Recall that by our assumption $c_3\geq c_7$. Hence,$${\rm deg}_{x_3}(v)={\rm deg}_{x_7}(v)< c_7\leq c_3.$$This yields that $(x_3x_7)v\in (I(G)^{\delta +1})_{\mathfrak{c}}$, a contradiction. This argument shows that ${\rm deg}_{x_7}(v)=c_7$. Thus, we proved that ${\rm deg}_{x_5}(v)=c_5$,  ${\rm deg}_{x_2}(v)=c_2$, ${\rm deg}_{x_4}(v)=c_4$ and ${\rm deg}_{x_7}(v)=c_7$. These equalities imply that $\delta=c_2+c_4+c_5+c_7$. Consequently, every monomial $u\in \Wc(\mathfrak{c},G)$ can be written as$$u=u'x_2^{c_2}x_4^{c_4}x_5^{c_5}x_7^{c_7},$$where $u'$ is a monomial of degree $c_2+c_4+c_5+c_7$ on $x_1, x_3, x_6$. Thus, Lemma \ref{polmat} implies that $\Wc(\mathfrak{c},G)$ enjoys the strong exchange property.

\medskip

{\bf (Case 2)} Suppose that there is a monomial $v\in \Wc(\mathfrak{c},G)$ with ${\rm deg}_{x_6}(v) < c_6$. Assume that $v=e_1\cdots e_{\delta}$, where $e_1, \ldots, e_{\delta}$ are edges of $G$. If ${\rm deg}_{x_5}(v) < c_5$, then $(x_5x_6)v\in (I(G)^{\delta +1})_{\mathfrak{c}}$, a contradiction. Therefore, ${\rm deg}_{x_5}(v)=c_5$. As$${\rm deg}_{x_6}(v)< c_6\leq c_5={\rm deg}_{x_5}(v),$$ we conclude that in the representation of $v=e_1\cdots e_{\delta}$, there is an edge, say, $e_1$, which is incident to $x_5$ but not to $x_6$. In other words, $e_1=\{x_1,x_5\}$. Let$$v'=x_6v/x_1=(x_5x_6)e_2\cdots e_{\delta}\in\Wc(\mathfrak{c},G).$$Since ${\rm deg}_{x_1}(v')< c_1$, we deduce from Case 1 that $\Wc(\mathfrak{c},G)$ is of Veronese type. 

\medskip

{\bf (Case 3)} Suppose that for each $v\in \Wc(\mathfrak{c},G)$, one has ${\rm deg}_{x_6}(v)=c_6$. This means that every $v\in \Wc(\mathfrak{c},G)$ is divisible by $(x_5x_6)^{c_6}$. Dividing each $v$ by $(x_5x_6)^{c_6}$ provides a correspondence between $\Wc(\mathfrak{c},G)$ and $\Wc(\mathfrak{c''},G-x_6)$, where$$\mathfrak{c''}=(c_1,c_2,c_3,c_4,c_5-c_6,c_7)\in \ZZ_{>0}^6.$$Since $G-x_6$ is the graph $K_{4,2}-M$, for a matching $M$ of $K_{4,2}$, we conclude from Theorem \ref{completeSEP} that $\Wc(\mathfrak{c},G)$ is of Veronese type. 
\end{proof}

\begin{Lemma} \label{c3threepend}
The finite graph $G$ on the vertex set $\{x_1, \ldots, x_6\}$ with the edge set$$E(G)=E(C_3)\cup\{\{x_1,x_4\},\{x_1,x_5\},\{x_2,x_6\}\},$$where $V(C_3)=\{x_1,x_2,x_3\}$, does not enjoy the strong exchange property.    
\end{Lemma}

\begin{proof}
Let $\mathfrak{c}=(1,1,1,1,1,1)\in \ZZ_{>0}^6$. Then $\delta_{\mathfrak{c}}(I(G))=2$. The monomials $$w_1=(x_1x_4)(x_2x_6), \, \, \, \, \, \, \, \, \, \, w_2=(x_1x_5)(x_2x_3)$$ belong to $\Wc(\mathfrak{c},G)$ with ${\rm deg}_{x_6}(w_1)> {\rm deg}_{x_6}(w_2)$ and ${\rm deg}_{x_5}(w_1)< {\rm deg}_{x_5}(w_2)$. If $\Wc(\mathfrak{c},G)$ enjoys the strong exchange property, then  $x_5w_1/x_6 \in \Wc(\mathfrak{c},G)$, which is impossible, as this monomial is divisible by $x_1x_4x_5$.
\end{proof}

\begin{Lemma} \label{c3pathpend}
The finite graph $G$ on the vertex set $\{x_1, \ldots, x_7\}$ with the edge set$$E(G)=E(C_3)\cup\{\{x_1,x_4\},\{x_4,x_5\},\{x_5,x_6\},\{x_2,x_7\}\},$$where $V(C_3)=\{x_1,x_2,x_3\}$,
does not enjoy the strong exchange property.    
\end{Lemma}

\begin{proof}
Let $\mathfrak{c}=(1,1,1,2,1,1,1)\in \ZZ_{>0}^7$. Then $\delta_{\mathfrak{c}}(I(G))=3$. The monomials $$w_1=(x_1x_4)(x_2x_7)(x_5x_6), \, \, \, \, \, \, \, \, \, \, w_2=(x_1x_4)(x_2x_3)(x_4x_5)$$ belong to $\Wc(\mathfrak{c},G)$ with
${\rm deg}_{x_7}(w_1)> {\rm deg}_{x_7}(w_2)$ and ${\rm deg}_{x_4}(w_1)< {\rm deg}_{x_4}(w_2)$. If $\Wc(\mathfrak{c},G)$ enjoys the strong exchange property, then  $x_4w_1/x_7 \in \Wc(\mathfrak{c},G)$, which is impossible, as this monomial is divisible by $x_4^2x_5x_6$.
\end{proof}

\begin{Lemma} \label{c3pathstar}
The finite graph $G$ on the vertex set $\{x_1, \ldots, x_7\}$ with the edge set$$E(G)=E(C_3)\cup\{\{x_1,x_4\},\{x_4,x_5\},\{x_5,x_6\},\{x_5,x_7\}\},$$where $V(C_3)=\{x_1,x_2,x_3\}$, does not enjoy the strong exchange property.    
\end{Lemma}

\begin{proof}
Let $\mathfrak{c}=(2,1,1,1,1,1,1)\in \ZZ_{>0}^7$. Then $\delta_{\mathfrak{c}}(I(G))=3$. The monomials $$w_1=(x_1x_2)(x_1x_4)(x_5x_6), \, \, \, \, \, \, \, \, \, \, w_2=(x_1x_4)(x_2x_3)(x_5x_7)$$ belong to $\Wc(\mathfrak{c},G)$ with ${\rm deg}_{x_1}(w_1)> {\rm deg}_{x_1}(w_2)$ and ${\rm deg}_{x_7}(w_1)< {\rm deg}_{x_7}(w_2)$. If $\Wc(\mathfrak{c},G)$ enjoys the strong exchange property, then  $x_7w_1/x_1 \in \Wc(\mathfrak{c},G)$, which is impossible, as this monomial is divisible by $x_5x_6x_7$.
\end{proof}

\begin{Lemma} \label{c3path4}
The finite graph $G$ on the vertex set $\{x_1, \ldots, x_7\}$ with edge set$$E(G)=E(C_3)\cup\{\{x_1,x_4\},\{x_4,x_5\},\{x_5,x_6\},\{x_6,x_7\}\},$$where $V(C_3)=\{x_1,x_2,x_3\}$,
does not enjoy the strong exchange property.    
\end{Lemma}

\begin{proof}
Let $\mathfrak{c}=(1,1,1,1,2,1,1)\in \ZZ_{>0}^7$. Then $\delta_{\mathfrak{c}}(I(G))=3$. The monomials $$w_1=(x_1x_2)(x_4x_5)(x_6x_7), \, \, \, \, \, \, \, \, \, \, w_2=(x_1x_3)(x_4x_5)(x_5x_6)$$ belong to $\Wc(\mathfrak{c},G)$ with ${\rm deg}_{x_2}(w_1)> {\rm deg}_{x_2}(w_2)$ and ${\rm deg}_{x_5}(w_1)< {\rm deg}_{x_5}(w_2)$. If $\Wc(\mathfrak{c},G)$ enjoys the strong exchange property, then  $x_5w_1/x_2 \in \Wc(\mathfrak{c},G)$, which is impossible, as this monomial is divisible by $x_5^2x_6x_7$.
\end{proof}

\begin{Lemma} \label{c3path3pend}
The finite graph $G$ on the vertex set $\{x_1, \ldots, x_7\}$ with the edge set$$E(G)=E(C_3)\cup\{\{x_1,x_4\},\{x_4,x_5\},\{x_5,x_6\},\{x_1,x_7\}\},$$where $V(C_3)=\{x_1,x_2,x_3\}$,
does not enjoy the strong exchange property.    
\end{Lemma}

\begin{proof}
Let $\mathfrak{c}=(1,1,2,1,1,1,1)\in \ZZ_{>0}^7$. Then $\delta_{\mathfrak{c}}(I(G))=3$. The monomials $$w_1=(x_1x_7)(x_2x_3)(x_5x_6), \, \, \, \, \, \, \, \, \, \, w_2=(x_1x_3)(x_2x_3)(x_4x_5)$$ belong to $\Wc(\mathfrak{c},G)$ with ${\rm deg}_{x_6}(w_1)> {\rm deg}_{x_6}(w_2)$ and ${\rm deg}_{x_3}(w_1)< {\rm deg}_{x_3}(w_2)$. If $\Wc(\mathfrak{c},G)$ enjoys the strong exchange property, then  $x_3w_1/x_6 \in \Wc(\mathfrak{c},G)$, which is impossible, as this monomial is divisible by $x_1x_3^2x_7$.
\end{proof}

\begin{Lemma} \label{c3path2}
The finite graph $G$ obtained from the triangle $C_3$ by attaching a path of length two to each of its vertices enjoys the strong exchange property.   
\end{Lemma}

\begin{proof}
Let $V(G)=\{x_1, \ldots, x_9\}$,
$V(C_3)=\{x_1,x_2,x_3\}$
and$$E(G)= E(C_3)
 \cup\{\{x_1, x_4\}, \{x_4,x_5\},\{x_2,x_6\},\{x_6,x_7\}, \{x_3,x_8\},\{x_8,x_9\}\}.$$
Fix $\mathfrak{c}=(c_1, \ldots, c_9)\in \ZZ_{>0}^9$. We show that $\Wc(\mathfrak{c},G)$ enjoys the strong exchange property. If  $c_5> c_4$, then for the vector $\mathfrak{c'}$ obtained from $\mathfrak{c}$ by replacing $c_5$ with $c_4$, one has $\Wc(\mathfrak{c},G)=\Wc(\mathfrak{c'},G)$. Hence, from the beginning we assume that $c_5\leq c_4$. Similarly, we suppose that $c_7\leq c_6$ and $c_9\leq c_8$. Set $\delta:=\delta_{\mathfrak{c}}(I(G))$. 

\medskip

{\bf (Case 1)} Suppose that there are $v\in \Wc(\mathfrak{c},G)$ and $1\leq i\leq 3$ with ${\rm deg}_{x_i}(v) < c_i$. By symmetry, we may assume that $i=1$. Thus, ${\rm deg}_{x_1}(v) < c_1$. Assume that $v=e_1\cdots e_{\delta}$, where $e_1, \ldots, e_{\delta}$ are edges of $G$. If ${\rm deg}_{x_2}(v) < c_2$, then $(x_1x_2)v$ belongs to $(I(G)^{\delta +1})_{\mathfrak{c}}$, a contradiction. Therefore, ${\rm deg}_{x_2}(v)=c_2$. Similarly, ${\rm deg}_{x_3}(v)=c_3$ and ${\rm deg}_{x_4}(v)=c_4$. Assume that ${\rm deg}_{x_7}(v)< c_7$. If in the representation of $v=e_1\cdots e_{\delta}$, there is an edge, say, $e_1$, which is equal to $\{x_2,x_6\}$, then$$(x_1x_7)v=(x_1x_2)(x_6x_7)e_2\cdots e_{\delta}\in (I(G)^{\delta +1})_{\mathfrak{c}},$$a contradiction. So, the edge $\{x_2,x_6\}$ does not appear in the representation of $v$. This implies that$${\rm deg}_{x_6}(v)={\rm deg}_{x_7}(v)< c_7\leq c_6.$$Consequently, $(x_6x_7)v\in (I(G)^{\delta +1})_{\mathfrak{c}}$, a contradiction. Thus ${\rm deg}_{x_7}(v)=c_7$. Similarly, ${\rm deg}_{x_9}(v)=c_9$. 

\smallskip

\noindent
{\bf (Subcase 1.1)} Suppose that in the representation of $v=e_1\cdots e_{\delta}$, there is an edge, say, $e_1$ which is equal to the edge $\{x_2,x_3\}$. If ${\rm deg}_{x_6}(v)< c_6$, then$$(x_1x_6)v=(x_1x_3)(x_2x_6)e_2\cdots e_{\delta}\in (I(G)^{\delta +1})_{\mathfrak{c}},$$a contradiction. Therefore, ${\rm deg}_{x_6}(v)=c_6$. Similarly, ${\rm deg}_{x_8}(v)=c_8$. Assume that ${\rm deg}_{x_5}(v)< c_5$. Since$${\rm deg}_{x_4}(v)=c_4\geq c_5> {\rm deg}_{x_5}(v),$$in the representation of $v=e_1\cdots e_{\delta}$, there is an edge, say, $e_2$ which is incident to $x_4$, but not to $x_5$. In other words, $e_2=\{x_1,x_4\}$. Then$$(x_1x_5)v=(x_1x_2)(x_1x_3)(x_4x_5)e_3\cdots e_{\delta}\in (I(G)^{\delta +1})_{\mathfrak{c}},$$a contradiction, which shows that ${\rm deg}_{x_5}(v)=c_5$. If ${\rm deg}_{x_1}(v)\leq c_1-2$, then$$x_1^2v=(x_1x_2)(x_1x_3)e_2\cdots e_{\delta}\in (I(G)^{\delta +1})_{\mathfrak{c}},$$a contradiction. Hence, ${\rm deg}_{x_1}(v)= c_1-1$. Thus, we showed that ${\rm deg}_{x_i}(v)=c_i$, for each $2\leq i \leq 9$ and ${\rm deg}_{x_1}(v)= c_1-1$. Consequently, $2\delta={\rm deg}(v)=(c_1+\cdots +c_9)-1$. Therefore, $\Wc(\mathfrak{c},G)$ enjoys the strong exchange property.

\smallskip

\noindent
{\bf (Subcase 1.2)} Suppose that the edge $\{x_2,x_3\}$ does not appear in the representation of $v$. Since $\{x_1,x_5,x_6,x_8\}$ is an independent set of $G$, it follows from our assumption that in the representation of $v=e_1\cdots e_{\delta}$, each $e_i$ is incident to exactly one of the vertices $x_2, x_3, x_4, x_7, x_9$. This yields that
\begin{eqnarray*}
\delta&=&{\rm deg}_{x_2}(v)+{\rm deg}_{x_3}(v)+{\rm deg}_{x_4}(v)+{\rm deg}_{x_7}(v)+{\rm deg}_{x_9}(v)\\
&=&c_2+c_3+c_4+c_7+c_9.   
\end{eqnarray*}
Now, let $u$ be an arbitrary monomial in $\Wc(\mathfrak{c},G)$. Again, using the fact that $\{x_1,x_5,x_6,x_8\}$ is an independent set of $G$, we conclude that$$u=u'x_2^{c_2}x_3^{c_3}x_4^{c_4}x_7^{c_7}x_9^{c_9},$$where $u'$ is a monomial of degree $c_2+c_3+c_4+c_7+c_9$ on $x_1, x_5, x_6, x_8$ with$${\rm deg}_{x_6}(u')\leq k_6:=\min\{c_6, c_2+c_7\}, \, \, \, \, \, \, \, \, \, \, {\rm deg}_{x_8}(u')\leq k_8:=\min\{c_8, c_3+c_9\}$$
$${\rm deg}_{x_1}(u')\leq k_1:=\min\{c_1, c_2+c_3+c_4\}.$$Moreover, as $x_7, x_9$ are leaves of $G$, with unique neighbors $x_6, x_8$, respectively, we deduce that $u'$ is divisible by $x_6^{c_7}x_8^{c_9}$. Thus, $$u=u''(x_6x_7)^{c_7}(x_8x_9)^{c_9}x_2^{c_2}x_3^{c_3}x_4^{c_4},$$where $u''$ is a $(k_1, c_5, k_6-c_7, k_8-c_9)$-bounded monomial of degree $c_2+c_3+c_4$ on  $x_1, x_5, x_6, x_8$. Conversely, it is easy to see for any $(k_1, c_5, k_6-c_7, k_8-c_9)$-bounded monomial of degree $c_2+c_3+c_4$ on $x_1, x_5, x_6, x_8$, one has $$w(x_6x_7)^{c_7}(x_8x_9)^{c_9}x_2^{c_2}x_3^{c_3}x_4^{c_4}\in\Wc(\mathfrak{c},G).$$This implies that the toric ring which is generated by the monomials belonging to  $\Wc(\mathfrak{c},G)$ is the algebra of Veronese type $$A(c_2+c_3+c_4;(k_1, c_5, k_6-c_7, k_8-c_9)),$$ 
Thus, in particular, it enjoys the strong exchange property, as desired.

\medskip

{\bf (Case 2)} Suppose that for every $v\in \Wc(\mathfrak{c},G)$ and each $i=1, 2, 3$, one has ${\rm deg}_{x_i}(v)=c_i$. Let $v=e_1\cdots e_{\delta}$ be an arbitrary monomial in  $\Wc(\mathfrak{c},G)$, where $e_1, \ldots, e_{\delta}$ are edges of $G$. Assume that ${\rm deg}_{x_5}(v)< c_5$. If ${\rm deg}_{x_4}(v)< c_4$, then $(x_4x_5)v$ belongs to $(I(G)^{\delta +1})_{\mathfrak{c}}$, a contradiction. Therefore, ${\rm deg}_{x_4}(v)=c_4$. This yields that$${\rm deg}_{x_4}(v)=c_4\geq c_5> {\rm deg}_{x_5}(v).$$Therefore, in the representation of $v=e_1\cdots e_{\delta}$, there is an edge, say, $e_1$ which is incident to $x_4$ but not to $x_5$. In other words, $e_1=\{x_1,x_4\}$. Consider the monomial$$v':=x_5v_/x_1=(x_4x_5)e_2\cdots e_{\delta}\in \Wc(\mathfrak{c},G),$$and note that ${\rm deg}_{x_1}(v')< c_1$. This contradicts our assumption. Consequently, ${\rm deg}_{x_5}(v)=c_5$. Similarly, ${\rm deg}_{x_7}(v)=c_7$ and ${\rm deg}_{x_9}(v)=c_9$. Therefore,$$v=v''x_1^{c_1}x_2^{c_2}x_3^{c_3}x_5^{c_5}x_7^{c_7}x_9^{c_9},$$where $v''$ is a monomial on $x_4, x_6, x_8$. As $v$ is an arbitrary monomial in $\Wc(\mathfrak{c},G)$, we deduce from Lemma \ref{polmat} that $\Wc(\mathfrak{c},G)$ enjoys the strong exchange property.
\end{proof}

\begin{Lemma} \label{c3path3}
The finite graph $G$ obtained from the triangle $C_3$ by attaching a path of length three to one of its vertices enjoys the strong exchange property.   
\end{Lemma}

\begin{proof}
Let $V(G)=\{x_1, \ldots, x_6\}$, $V(C_3)=\{x_1,x_2,x_3\}$ and$$E(G)= E(C_3)
 \cup\{\{x_1, x_4\}, \{x_4,x_5\},\{x_5,x_6\}\}.$$Fix $\mathfrak{c}=(c_1, \ldots, c_6)\in \ZZ_{>0}^6$. We show that $\Wc(\mathfrak{c},G)$ enjoys the strong exchange property. If  $c_6 > c_5$, then for the vector $\mathfrak{c'}$ obtained from $\mathfrak{c}$ by replacing $c_6$ with $c_5$, one has $\Wc(\mathfrak{c},G)=\Wc(\mathfrak{c'},G)$. Thus, we may assume that $c_6\leq c_5$. If  $c_5 > c_4+c_6$, then for the vector $\mathfrak{c''}$ obtained from $\mathfrak{c}$ by replacing $c_5$ with $c_4+c_6$, one has $\Wc(\mathfrak{c},G)=\Wc(\mathfrak{c''},G)$. Therefore, we may assume that $c_5\leq c_6+c_4$. Set 
 $\delta:=\delta_{\mathfrak{c}}(I(G))$. 

\medskip

{\bf (Case 1)} Suppose that there is $v\in \Wc(\mathfrak{c},G)$ and  with ${\rm deg}_{x_1}(v) < c_1$. Assume that $v=e_1\cdots e_{\delta}$, where $e_1, \ldots, e_{\delta}$ are edges of $G$. If ${\rm deg}_{x_2}(v) < c_2$, then $(x_1x_2)v$ belongs to $(I(G)^{\delta +1})_{\mathfrak{c}}$, a contradiction. Thus, ${\rm deg}_{x_2}(v)=c_2$. Similarly, ${\rm deg}_{x_3}(v)=c_3$ and ${\rm deg}_{x_4}(v)=c_4$. Assume that ${\rm deg}_{x_6}(v)< c_6$. If in the representation of $v=e_1\cdots e_{\delta}$, there is an edge, say, $e_1$, which is equal to $\{x_4,x_5\}$, then$$(x_1x_6)v=(x_1x_4)(x_5x_6)e_2\cdots e_{\delta}\in (I(G)^{\delta +1})_{\mathfrak{c}},$$a contradiction. So,  $\{x_4,x_5\}$ does not appear in the representation of $v$ and consequently,$${\rm deg}_{x_5}(v)={\rm deg}_{x_6}(v)< c_6\leq c_5.$$It follows that $(x_5x_6)v\in (I(G)^{\delta +1})_{\mathfrak{c}}$, a contradiction. Thus, ${\rm deg}_{x_6}(v)=c_6$.  

\smallskip

\noindent
{\bf (Subcase 1.1)} Suppose that in the representation of $v=e_1\cdots e_{\delta}$, there is an edge, say $e_1$ which is equal to the edge $x_2x_3$. Assume that ${\rm deg}_{x_5}(v)< c_5$. Recall from the first paragraph of the proof that $c_5\leq c_6+c_4$. This yields that$${\rm deg}_{x_4}(v)+{\rm deg}_{x_6}(v)=c_4+c_6\geq c_5 > {\rm deg}_{x_5}(v).$$ Consequently, in the representation of $v=e_1\cdots e_{\delta}$, there is an edge, say, $e_2$ which is incident to either $x_4$ or $x_6$, but not to $x_5$. By the structure of $G$, we must have $e_2=\{x_1,x_4\}$. Then$$(x_1x_5)v=(x_1x_2)(x_1x_3)(x_4x_5)e_3\cdots e_{\delta}\in (I(G)^{\delta +1})_{\mathfrak{c}},$$a contradiction, which shows that ${\rm deg}_{x_5}(v)=c_5$. If ${\rm deg}_{x_1}(v)\leq c_1-2$, then$$x_1^2v=(x_1x_2)(x_1x_3)e_2\cdots e_{\delta}\in (I(G)^{\delta +1})_{\mathfrak{c}},$$a contradiction. Hence, ${\rm deg}_{x_1}(v)= c_1-1$. Thus, we showed that ${\rm deg}_{x_i}(v)=c_i$ for each $2 \leq i \leq 6$ and ${\rm deg}_{x_1}(v)= c_1-1$. Therefore, $2\delta={\rm deg}(v)=(c_1+\cdots +c_6)-1$. So, $\Wc(\mathfrak{c},G)$ enjoys the strong exchange property.

\smallskip

\noindent
{\bf (Subcase 1.2)} Suppose that the edge $\{x_2,x_3\}$ does not appear in the representation of $v$. Since $\{x_1,x_5\}$ is an independent set of $G$, it follows from our assumption that in the representation of $v=e_1\cdots e_{\delta}$, each $e_i$ is incident to exactly one of the vertices $x_2, x_3, x_4, x_6$. Hence,$$\delta={\rm deg}_{x_2}(v)+{\rm deg}_{x_3}(v)+{\rm deg}_{x_4}(v)+{\rm deg}_{x_6}(v)=c_2+c_3+c_4+c_6.$$Now, let $u$ be an arbitrary monomial in $\Wc(\mathfrak{c},G)$. Again, using the fact that $\{x_1,x_5\}$ is an independent set of $G$, we conclude that$$u=u'x_2^{c_2}x_3^{c_3}x_4^{c_4}x_6^{c_6},$$where $u'$ is a monomial of degree $c_2+c_3+c_4+c_6$ on $x_1, x_5$. Thus, Lemma \ref{polmat} implies that $\Wc(\mathfrak{c},G)$ enjoys the strong exchange property.

\medskip

{\bf (Case 2)} Suppose that for every $v\in \Wc(\mathfrak{c},G)$, we have ${\rm deg}_{x_1}(v)=c_1$. Let $v=e_1\cdots e_{\delta}$ be an arbitrary monomial belonging to $\Wc(\mathfrak{c},G)$, where $e_1, \ldots, e_{\delta}$ are edges of $G$. Assume that ${\rm deg}_{x_5}(v)< c_5$. If ${\rm deg}_{x_4}(v)< c_4$, then $(x_4x_5)v\in (I(G)^{\delta +1})_{\mathfrak{c}}$, a contradiction. Therefore, ${\rm deg}_{x_4}(v)=c_4$. Similarly, ${\rm deg}_{x_6}(v)=c_6$. Then the same argument as in Subcase 1.1 implies that in the representation of $v=e_1\cdots e_{\delta}$, there is an edge, say $e_1$ which is equal to $\{x_1,x_4\}$. Consider the monomial$$v':=x_5v/x_1=(x_4x_5)e_2\cdots e_{\delta}\in \Wc(\mathfrak{c},G),$$and note that ${\rm deg}_{x_1}(v')< c_1$, a contradiction.  Thus, ${\rm deg}_{x_5}(v)=c_5$. 

Without loss of generality, we may assume that $c_3\leq c_2$. Suppose that ${\rm deg}_{x_3}(v)< c_3$. If the edge $\{x_1,x_2\}$ appears in the representation of $v$, then replacing this edge with $\{x_2,x_3\}$, we deduce that $v''=x_3v/x_1\in \Wc(\mathfrak{c},G)$ and ${\rm deg}_{x_1}(v'')< c_1$, a contradiction. Therefore, the edge $\{x_1,x_2\}$ does not appear in the representation of $v$. Consequently,$${\rm deg}_{x_2}(v)\leq {\rm deg}_{x_3}(v)< c_3\leq c_2,$$which is a contradiction, as $(x_2x_3)v\in (I(G)^{\delta +1})_{\mathfrak{c}}$. Hence, ${\rm deg}_{x_3}(v)=c_3$ and $$v=wx_1^{c_1}x_3^{c_3}x_5^{c_5},$$where $w$ is a monomial on $x_2, x_4, x_6$. Since $v$ is an arbitrary monomial in $\Wc(\mathfrak{c},G)$, we deduce from Lemma \ref{polmat} that $\Wc(\mathfrak{c},G)$ enjoys the strong exchange property.    
\end{proof}

\begin{Lemma} \label{c3paths2}
The finite graph $G$ obtained from the triangle $C_3$ by attaching a finite number of paths of length two to one of its vertices enjoys the strong exchange property. 
\end{Lemma}

\begin{proof}
Let the number of paths attached to one of the vertices of $C_3$ be $k$. Let $V(G)=\{x_1, \ldots, x_{2k+3}\}$, $V(C_3)=\{x_1,x_2,x_3\}$ and$$E(G)= E(C_3)
 \cup\big\{\{x_1, x_{3+i}\}\mid 1\leq i\leq k\big\}\cup,\big\{\{x_i, x_{i+k}\}\mid 4\leq i\leq k+3\big\}.$$
 Fix $\mathfrak{c}=(c_1, \ldots, c_{2k+3})\in \ZZ_{>0}^{2k+3}$. We show that $\Wc(\mathfrak{c},G)$ enjoys the strong exchange property. If  $c_{i+k} > c_i$, for some $i$ with $4\leq i\leq k+3$, then for the vector $\mathfrak{c'}$ obtained from $\mathfrak{c}$ by replacing $c_{i+k}$ with $c_i$, one has $\Wc(\mathfrak{c},G)=\Wc(\mathfrak{c'},G)$. Hence, we may assume that $c_{i+k}\leq c_k$, for each integer $i$ with $4\leq i\leq k_3$. Set 
 $\delta:=\delta_{\mathfrak{c}}(I(G))$. 

\medskip

{\bf (Case 1)} Suppose that there is $v\in \Wc(\mathfrak{c},G)$ and  with ${\rm deg}_{x_1}(v) < c_1$. Assume that $v=e_1\cdots e_{\delta}$, where $e_1, \ldots, e_{\delta}$ are edges of $G$. If ${\rm deg}_{x_2}(v) < c_2$, then $(x_1x_2)v$ belongs to $(I(G)^{\delta +1})_{\mathfrak{c}}$, a contradiction. Therefore, ${\rm deg}_{x_2}(v)=c_2$. Similarly, ${\rm deg}_{x_3}(v)=c_3$ and ${\rm deg}_{x_i}(v)=c_i$, for each $i$ with $4\leq i\leq k+3$.   

\smallskip

\noindent
{\bf (Subcase 1.1)} Suppose that in the representation of $v=e_1\cdots e_{\delta}$, there is an edge, say, $e_1$ which is equal to $\{x_2, x_3\}$. Assume that ${\rm deg}_{x_{k+4}}(v)< c_{k+4}$. Recall from the first paragraph of the proof that $c_{k+4}\leq c_4$. This yields that$${\rm deg}_{x_4}(v)=c_4\geq c_{k+4} > {\rm deg}_{x_{k+4}}(v).$$Consequently, in the representation of $v=e_1\cdots e_{\delta}$, there is an edge, say $e_2$ which is incident to $x_4$, but not to $x_{k+4}$. By the structure of $G$, one has $e_2=\{x_1,x_4\}$. Then$$(x_1x_{k+4})v=(x_1x_2)(x_1x_3)(x_4x_{k+4})e_3\cdots e_{\delta}\in (I(G)^{\delta +1})_{\mathfrak{c}},$$a contradiction, which shows that ${\rm deg}_{x_{k+4}}(v)=c_{k+4}$. Similarly, we deduce that ${\rm deg}_{x_{k+i}}(v)=c_{k+i}$, for each $i$ with $4\leq i\leq k+3$. If ${\rm deg}_{x_1}(v)\leq c_1-2$, then$$x_1^2v=(x_1x_2)(x_1x_3)e_2\cdots e_{\delta}\in (I(G)^{\delta +1})_{\mathfrak{c}},$$a contradiction. Hence, ${\rm deg}_{x_1}(v)= c_1-1$. Thus, one has  ${\rm deg}_{x_i}(v)=c_i$ for each $2 \leq i\leq 2k+3$ and ${\rm deg}_{x_1}(v)= c_1-1$. Therefore, $2\delta={\rm deg}(v)=(c_1+\cdots +c_{2k+3})-1$. So, $\Wc(\mathfrak{c},G)$ enjoys the strong exchange property.

\smallskip

\noindent
{\bf (Subcase 1.2)} Suppose that the edge $\{x_2,x_3\}$ does not appear in the representation of $v$. Since $\{x_1,x_{k+4}, \ldots, x_{2k+3}\}$ is an independent set of $G$, it follows that in the representation of $v=e_1\cdots e_{\delta}$, each $e_i$ is incident to exactly one of the vertices $x_2, x_3, x_4, \ldots, x_{k+3}$. This yields that
\begin{eqnarray*}
 \delta&=&{\rm deg}_{x_2}(v)+{\rm deg}_{x_3}(v)+{\rm deg}_{x_4}(v)+\cdots +{\rm deg}_{x_{k+3}}(v)
 \\
 &=&c_2+c_3+c_4+\cdots + c_{k+3}.   
\end{eqnarray*}
Now, let $u$ be an arbitrary monomial in $\Wc(\mathfrak{c},G)$. Again, using the fact that $\{x_1,x_{k+4}, \ldots, x_{2k+3}\}$ is an independent set of $G$, we conclude that$$u=u'x_2^{c_2}x_3^{c_3}x_4^{c_4}\cdots x_{k+3}^{c_{k+3}},$$where $u'$ is a $(c_1, c_{k+4}, \ldots, c_{2k+3})$-bounded monomial of degree $c_2+c_3+c_4+\cdots + c_{k+3}$ on variables $x_1, x_{k+4}, \ldots x_{2k+3}$. Conversely, as $c_{i+k}\leq c_k$, for each $i$ with $4\leq i\leq k+3$, one can easily see that for an arbitrary $(c_1, c_{k+4}, \ldots, c_{2k+3})$-bounded monomial $w$ of degree $c_2+c_3+c_4+\cdots + c_{k+3}$ on variables $x_1, x_{k+4}, \ldots x_{2k+3}$, one has$$u'x_2^{c_2}x_3^{c_3}x_4^{c_4}\cdots x_{k+3}^{c_{k+3}}\in \Wc(\mathfrak{c},G).$$This implies that the toric ring which is generated by the monomials belonging to $\Wc(\mathfrak{c},G)$ is the algebra of Veronese type $$A(c_2+c_3+c_4+\cdots + c_{k+3};(c_1, c_{k+4}, \ldots, c_{2k+3})).$$ In particular, $\Wc(\mathfrak{c},G)$ enjoys the strong exchange property.

\medskip

{\bf (Case 2)} Suppose that each $v\in \Wc(\mathfrak{c},G)$ satisfies ${\rm deg}_{x_1}(v)=c_1$. 

\medskip

{\bf Claim.} Assume that $v=e_1\cdots e_{\delta}$ and $v'=e_1'\cdots e_{\delta}'$ belong to $\Wc(\mathfrak{c},G)$, where $e_1, \ldots, e_{\delta}, e_1', \ldots, e_{\delta}'$ are edges of $G$. Then$$|\{i: e_i=\{x_2,x_3\}\}|=|\{i: e_i'=\{x_2,x_3\}\}|.$$

{\it Proof of the claim.} By assumption, ${\rm deg}_{x_1}(v)=c_1$. Assume that ${\rm deg}_{x_{k+4}}(v)< c_{k+4}$. If ${\rm deg}_{x_4}(v)< c_4$, then $(x_4x_{k+4})v\in (I(G)^{\delta +1})_{\mathfrak{c}}$, a contradiction. Therefore, ${\rm deg}_{x_4}(v)=c_4$. Consequently,$${\rm deg}_{x_4}(v)=c_4\geq c_{k+4}> {\rm deg}_{x_{k+4}}(v).$$Thus, in the representation of $v=e_1\cdots e_{\delta}$, there is an edge, say, $e_{\delta}$ which is incident to $x_4$ but not to $x_{k+4}$. By the structure of $G$, one has $e_{\delta}=\{x_1,x_4\}$. Let$$v'':=x_{k+4}v/x_1=(x_4x_{k+4})e_1\cdots e_{\delta-1}\in \Wc(\mathfrak{c},G),$$and note that ${\rm deg}_{x_1}(v'')< c_1$. This contradicts our assumption. Consequently, ${\rm deg}_{x_{k+4}}(v)=c_{k+4}$. Similarly, ${\rm deg}_{x_{k+i}}(v)=c_{k+i}$, for each $4 \leq i \leq k+3$.

Set $a:=|\{i: e_i=\{x_2,x_3\}\}|$ and assume that the edges $e_1, \ldots, e_a$ are equal to $\{x_2,x_3\}$. It follows that each of the edges $e_{a+1}, \ldots, e_{\delta}$ are incident to exactly one of the vertices $x_1, x_{k+4}, \ldots x_{2k+3}$. As a result,
\begin{eqnarray*}
 \delta&=&a+{\rm deg}_{x_1}(v)+{\rm deg}_{x_{k+4}}(v)+\cdots + {\rm deg}_{x_{2k+3}}(v)\\
 &=&a+c_1+c_{k+4}+\cdots + c_{2k+3}.   
\end{eqnarray*}
Hence,$$|\{i: e_i=x_2x_3\}|=a=\delta-(c_1+c_{k+4}+\cdots + c_{2k+3}).$$Similarly,$$|\{i: e_i'=\{x_2,x_3\}\}|=\delta-(c_1+c_{k+4}+\cdots + c_{2k+3}).$$This proves the claim.

\medskip

It follows from the claim that there is an integer $a\geq 0$ such that, for each monomial $e_1\ldots e_{\delta}\in \Wc(\mathfrak{c},G)$, one has $$|\{i: e_i=\{x_2,x_3\}\}|=a.$$Therefore,$$\Wc(\mathfrak{c},G)=\{(x_2x_3)^au \mid u\in \Wc(\mathfrak{c'},G')\},$$where $G'$ is the finite graph obtained from $G$ by deleting the edge $\{x_2,x_3\}$ and $\mathfrak{c'}$ is the vector obtained from $\mathfrak{c}$ by replacing $c_2$ and $c_3$ with $c_2-a$ and $c_3-a$, respectively. It follows from Theorem \ref{treeclass} that $G'$ enjoys the strong exchange property. Hence, $\Wc(\mathfrak{c},G)$ enjoys the strong exchange property as well. 
\end{proof}

Finally, we come to the classification of unicyclic graphs which enjoy the strong exchange property.

\begin{Theorem}
\label{classification_unicyclic}
Let $G$ be a unicyclic graph and $\ell\geq 3$ the length of the unique cycle of $G$.
\begin{itemize}
\item[(i)] If $\ell\geq 8$, then $G$ does not enjoy the strong exchange property. 

\item[(ii)] If $\ell\in\{5,6,7\}$, then $G$ enjoys the strong exchange property if and only if the independence number of $G$ is at most three. 
\item[(iii)] If $\ell=4$, then $G$ enjoys the strong exchange property if and only if
\begin{itemize}
\item [(1)] $G$ is obtained from $C_4$ by attaching at most one pendant edge to each of its vertices; or
\item [(2)] $G$ is the graph described in Lemma \ref{c4pathpendad}; or
\item [(3)] $G$ is obtained from $C_4$ by attaching a path of length two to one of its vertices.
\end{itemize}
\item[(iv)] If $\ell=3$, then $G$ enjoys the strong exchange property if and only if
\begin{itemize}
\item [(1)] $G$ is obtained from $C_3$ by attaching at most one path of length at most two to each of its vertices; or
\item [(2)] $G$ is obtained from $C_3$ by attaching a path of length three to one of its vertices; or
\item [(3)] $G$ is obtained from $C_3$ by attaching a finite number of paths of length at most two to one of its vertices.
\end{itemize}
\end{itemize}
\end{Theorem}

\begin{proof}
Let $C$ denote the unique cycle of $G$.

(i) Suppose that $G$ enjoys the strong exchange property.  Then repeated applications of Lemma \ref{treedel} guarantee that $C$ enjoys the strong exchange property, which contradicts Theorem \ref{maincyc}.

(ii) The ``if'' part follows from Lemma \ref{trianind}. To prove the ``only if'' part, first consider the case $\ell=7$. If $G\neq C_7$ enjoys the strong exchange property, then repeated applications of Lemma \ref{treedel} guarantee that the finite graph  described in Lemma \ref{c7pend} enjoys the strong exchange property, a contradiction. By a similar argument, in the case $\ell=6$, we must have $G=C_6$. the only difference is that one needs to use Lemma \ref{c6pend} instead of Lemma \ref{c7pend}.

Let $\ell=5$.   If $G$ enjoys the strong exchange property and  if the independence number of $G$ is at least four, then the repeated applications of Lemma \ref{treedel} guarantees that one of the  graphs described in Lemmas \ref{c5pendad}, \ref{c5pendnoad}, \ref{c5twopend}, \ref{c5path} and \ref{c5star} must enjoy the strong exchange property, a contradiction.

(iii) First we prove the ``if'' part. By using Lemmas \ref{c4pendall} and \ref{treedel}, we conclude that the finite graph described in (1) enjoys the strong exchange property. Since the finite graph described in (3) has independence number three, it enjoys the strong exchange property by Lemma \ref{trianind}.  

Now we prove the ``only if'' part.  Suppose that $G$ enjoys the strong exchange property and that $G$ coincides with none of the finite graphs described in (1), (2) and (3). Then by the repeated applications of Lemma \ref{treedel}, we conclude that one of the finite graphs described in Lemmas \ref{c4twopend}, \ref{c4pendpath}, \ref{c4twopath}, \ref{c4star} and \ref{c4tpathlong} must enjoy the strong exchange property, a contradiction.

(iv) The ``if'' part follows from Lemmas \ref{c3path2}, \ref{c3path3}  and \ref{c3paths2} together with the repeated applications of Lemma \ref{treedel}. To prove the ``only if'' part, suppose that $G$ enjoys the strong exchange property and that $G$ coincides with none of the graphs described in (1), (2) and (3). Again, Lemma \ref{treedel} guarantees that one of the finite graphs described in Example \ref{Ex_triangle} and Lemmas \ref{c3threepend}, \ref{c3pathpend}, \ref{c3pathstar}, \ref{c3path4} and \ref{c3path3pend} must enjoy the strong exchange property, a contradiction.
\end{proof}

As was said in Introduction, with taking into account of the most attractive research problems \cite[p.~241]{HH_discrete}, one can naturally ask if, for all finite graphs $G$ on $n$ vertices and for all $\mathfrak{c} \in \ZZ_{>0}^n$, the toric ideal $\Ker(\pi_G^\mathfrak{c})$ possesses a quadratic Gr\"obner basis and is generated by all symmetric exchange binomials of $\Ker(\pi_G^\mathfrak{c})$.  

The study done in the present paper especially encourages the authors to propose the following   

\begin{Conjecture}
\label{conjecture}
If $G$ is a unicyclic graph on $n$ vertices, then for all $\mathfrak{c} \in \ZZ_{>0}^n$, the toric ideal $\Ker(\pi_G^\mathfrak{c})$ is generated by all symmetric exchange binomials of $\Ker(\pi_G^\mathfrak{c})$. 
\end{Conjecture}

We conclude the present paper with 

\begin{Example}
We work in the situation of Lemma \ref{c3pathpend}.  The toric ring $\Bc(\mathfrak{c},G)$ is generated by
\begin{eqnarray*}
w_1=x_1x_2x_3x_4^2x_5, & 
w_2=x_1x_2x_3x_4x_5x_6, &  
w_3=x_1x_2x_3x_4x_5x_7,
\\
w_4=x_1x_2x_3x_5x_6x_7, &  
w_5=x_1x_2x_4^2x_5x_7, &  
w_6=x_1x_2x_4x_5x_6x_7  
\end{eqnarray*}
and its toric ideal $\Ker(\pi_G^\mathfrak{c})$ is generated by the symmetric exchange binomials 
\begin{eqnarray*}
z_4z_5-z_3z_6, \, \, \, \, \,
z_2z_3-z_1z_4, \, \, \, \, \,
z_2z_5-z_1z_6.
\end{eqnarray*}
\end{Example}

\section*{Acknowledgments}
The second author is supported by a FAPA grant from Universidad de los Andes.

\section*{Statements and Declarations}
The authors have no Conflict of interest to declare that are relevant to the content of this article.

\section*{Data availability}
Data sharing does not apply to this article as no new data were
created or analyzed in this study.

\end{document}